\documentclass{amsart}
\usepackage{amssymb,amscd}
\usepackage{graphicx}

\usepackage{latexsym}
\usepackage{multirow}
\usepackage{setspace}

\usepackage{amsthm}
\newtheorem{theorem}{Theorem}

\newtheorem{lemma}{Lemma}[section]
\newtheorem{cor}{Corollary}
\newtheorem{prop}{Proposition}[section]

\newcommand{\res}[1]{\begin{array}[d]{l}\\{\rm Res}\\^{#1}\end{array}\hspace{-1mm}}

\newcommand{\bc}{\mathbb{C}}

\newcommand{\bp}{\mathbb{P}}

\newcommand{\br}{\mathbb{R}}
\newcommand{\bz}{\mathbb{Z}}
\newcommand{\modm}{\mathcal{M}}

\newcommand{\dd}{\mathcal{D}}
\newcommand{\f}{\mathcal{F}}

\newcommand{\fat}{\f\hspace{-.3mm}{\rm at}_{g,n}}

\begin{document}

\title{Polynomials representing Eynard-Orantin invariants}
\author{Paul Norbury and Nick Scott}
\address{Department of Mathematics and Statistics\\
University of Melbourne\\Australia 3010}
\email{pnorbury@ms.unimelb.edu.au\\ N.Scott@ms.unimelb.edu.au}
\date{}

\begin{abstract}

\noindent The Eynard-Orantin invariants of a plane curve are multilinear differentials on the curve.  For a particular class of genus zero plane curves these invariants can be equivalently expressed in terms of simpler expressions given by polynomials obtained from an expansion of the Eynard-Orantin invariants around a point on the curve.  This class of curves contains many interesting examples.

\end{abstract}

\maketitle

\section{Introduction}

As a tool for studying enumerative problems in geometry Eynard and Orantin \cite{EOrInv} associate multilinear differentials to any compact Riemann surface $C$ equipped with two meromorphic functions $x$ and $y$ with the property that the branch points of $x$ are simple and the map
\[ \begin{array}[b]{rcl} C&\to&\bc^2\\p&\mapsto& (x(p),y(p))\end{array}\]
is an immersion.  For every $(g,n)\in\bz^2$ with $g\geq 0$ and $n>0$ they define a multilinear differential, i.e. a tensor product of meromorphic 1-forms on the product $C^n$, notated by $\omega^g_n(p_1,...,p_n)$ for $p_i\in C$.  When $2g-2+n>0$, $\omega^g_n(p_1,...,p_n)$ is defined recursively in terms of local  information around branch points of $x$ of $\omega^{g'}_{n'}(p_1,...,p_n)$ for $2g'-2+n'<2g-2+n$.  A dilaton relation between $\omega^g_{n+1}$ and $\omega^g_n$ can be applied in the $n=0$ case to define $F^g$ ($=$ ``$\omega^g_0$") for $g\geq 2$, known as the symplectic invariants of the curve.  The invariant $F^g$ recursively uses all $\omega^{g'}_n$ with $g'+n\leq g+1$.

In principle $F^g$ and $\omega^g_n$ can be calculated explicitly from the recursion relations defining them, and implemented on a computer.  In practice, the expressions obtained this way are unwieldy and computable only for small $g$ and $n$.  The main aim of this paper is to express $\omega^g_n$ in a simpler form---essentially its inverse discrete Laplace transform---and to develop methods to calculate general formulae for $F^g$ in some examples. 

We consider the Eynard-Orantin invariants of the class of genus zero curves for which $x$ is two-to-one and has two branch points.  The case of one branch point $x=z^2$ is dealt with in \cite{EOrInv}.  One can parametrise the domain so that the curve can be written:
\begin{equation}  \label{eq:curve}
C=\begin{cases}x=a+b(z+1/z)\\ 
y=y(z)
\end{cases}
\end{equation}
for constants $a$, $b$ and any rational function $y(z)$ with $y'(\pm1)\neq 0$.  The definition of $\omega^g_n$ and its properties requires  the Riemann surface to be compact.  Nevertheless, by taking sequences of compact Riemann surfaces one can extend the definition to allow $y(z)$ to be any analytic function defined on a domain in $\bc$ containing $z=\pm1$, e.g. $y(z)=\ln(z)$ defined on the complement of the negative imaginary axis.

\begin{theorem}  \label{th:main}
For the plane curve $C$ defined in (\ref{eq:curve}) and $2g-2+n>0$, $\omega_n^g(z_1,...,z_n)$ has an expansion around $\{z_i=0\}$ given by
\begin{equation}  \label{eq:expand}
\omega_n^g(z_1,..,z_n)=\frac{d}{dz_1}\dots\frac{d}{dz_n}\sum_{b_i>0}N^g_n(b_1,\dots,b_n)z_1^{b_1}\dots z_n^{b_n} dz_1\dots dz_n
\end{equation}
where $N^g_n$ is a quasi-polynomial in the $b_i^2$ of homogeneous degree $3g-3+n$, dependent on the parity of the $b_i$ and is symmetric in all variables of the same parity.
\end{theorem}
The parity dependence means that there exists polynomials $N^g_{n,k}(b_1,...,b_n)$ for $k=1,...,n$ such that $N^g_n(b_1,...,b_n)$ decomposes
\[ N^g_n(b_1,...,b_n)=N^g_{n,k}(b_1,...,b_n),\quad k={\rm number\ of\ odd\ }b_i.\]

The polynomials representing $N^g_n$ are simpler expressions than $\omega_n^g(z_1,...,z_n)$ and can have meaning themselves.  When $x=z+1/z$ and $y=z$, $N^g_n$ arises as a solution of a Hurwitz problem \cite{NorCou, NorStr}, with typical expression $N_{4,0}^{0}=\frac{1}{4}(b_1^2+b_2^2+b_3^2+b_4^2)-1$ (while the much larger expression $\omega^0_4$ can be expressed as the sum of 32 rational functions.)  The case $x=z+1/z$ and $y=\ln{z}$ arises when studying partitions and the Plancherel measure \cite{EOrAlg}, with $N_{4,0}^{0}=\frac{1}{4}\left(b_1^2+b_2^2+b_3^2+b_4^2\right)$.

\begin{theorem}  \label{th:inter}
The coefficients of the top homogeneous degree terms in the polynomial $N^g_{n,k}(b_1,...,b_n)$, defined above, can be expressed in terms of intersection numbers of tautological line bundles over the moduli space $L_i\to{\modm}_{g,n}$.  For $\sum_i\beta_i=3g-3+n$, the coefficient $v_{\bf \beta}$ of $\prod b_i^{2\beta_i}$ is
\[v_{\bf \beta}=\frac{y'(1)^{2-2g-n}+(-1)^ky'(-1)^{2-2g-n}}{x''(1)^{2g-2+n}2^{3g-3+n}\beta!}\int_{\overline{\modm}_{g,n}}c_1(L_1)^{\beta_1}...c_1(L_n)^{\beta_n}\]

\end{theorem}

The invariants $\omega^g_n$ satisfy string and dilaton equations---see Section~\ref{sec:string}---which enable one to express $N^g_{n+1}$ recursively in terms of $N^g_n$.   They are most easily expressed when $y(z)$ is a monomial.  The following two theorems apply to the polynomials $N^g_{n,k}$ although we abuse notation and simply write $N^g_n$.

\begin{theorem}  \label{th:string0}
The polynomials associated to the plane curves 
\[x=z+1/z,\quad y=z^{m}/m,\quad m=1,2,...\]
satisfy the following recursion relations.
\begin{equation}  \label{eq:strmon}
N^g_{n+1}(m,b_1,\dots b_n) =\sum_{j=1}^n\sum_{k=1}^{b_j}kN^g_n(b_1,\dots,b_n)|_{b_j=k}
\end{equation}
\begin{align} \label{eq:strmon1} (m+1)N^g_{n+1}(m&+1,b_1,\dots,b_n)+(m-1)N^g_{n+1}(m-1,b_1,\dots,b_n)\\
&=2m\sum_{j=1}^n\sum_{k=1}^{b_j}kN^g_n(b_1,\dots,b_n)|_{b_j=k}-m\sum_{j=1}^nb_jN^g_n(b_1,\dots, b_n)\nonumber\end{align}
\begin{align}  \label{eq:strmon2} (m-2&)N^g_{n+1}(m-2, b_1,\dots,b_n)-2mN^g_{n+1}(m, b_1,\dots,b_n)\\
&+(m+2)N^g_{n+1}(m+2, b_1,\dots,b_n)
=m\sum_{j=1}^n\sum_{k=1\pm b_j}kN^g_n(b_1,\dots,b_n)|_{b_j=k}\nonumber
\end{align}
\begin{equation}  \label{eq:dilmon}
N^g_{n+1}(m+1,b_1,...,b_n)-N^g_{n+1}(m-1,b_1,...,b_n)=m(2g-2+n)N^g_n(b_1,...,b_n)
\end{equation} 
\end{theorem}
\begin{cor}  
The symplectic invariants of the curve $x=z+1/z$, $y=z^{m}/m$ satisfy
\[F^g=\frac{1}{m(2g-2)}\left(N^g_{1}(m+1)-N^g_{1}(m-1)\right)\]
for $g\geq 2$.
\end{cor}
The polynomials corresponding to $y=\ln{z}$ satisfy recursions obtained by setting $m=0$ into (\ref{eq:strmon}) and $((\ref{eq:strmon1})-(\ref{eq:dilmon}))/m$.  

Theorem~\ref{th:string0} is a special case of the following more general theorem that applies to any analytic function $y(z)$ defined on a domain in $\bc$ containing $z=\pm1$ expanded as $y(z)\sim\sum (a_k+zb_k)(1-z^2)^k$.  For example $y(z)=\ln{z}\sim\sum\frac{(1-z^2)^k}{-2k}$.  

First we need the following notation.  Define the operator $\dd$ on functions by
\[\dd f(n)=f(n+1).\]
Further, define $\dd\{f(n)\}_{n=a}=f(a+1)$.  As usual, for a polynomial $p(z)=\sum p_iz^i$, define $p(\dd)\{f(n)\}_{n=a}=\sum p_if(a+i)$. Note that $\dd-I$ is a discrete derivative and $y(\dd)\sim\sum (a_k+zb_k)(1-\dd^2)^k$ is a sum over powers of the discrete derivative $\dd^2-I$.  Put $b_S=(b_1,...,b_n)$ and $|b_S|=b_1+...+b_n$. 
\begin{theorem}  \label{th:string}
For the plane curve $x=z+1/z$, $y=y(z)$
\begin{align}
\dd y(\dd)\left\{mN^g_{n+1}(m,b_S)\right\}_{m=-1}\hspace{-1mm}&=\sum_{j=1}^n\hspace{-3mm}\sum_{\tiny\begin{array}{c}k=1\\k\not\equiv b_j(2)\end{array}}^{b_j}\hspace{-3mm}kN^g_n(b_S)|_{b_j=k}\label{eq:str1}\\
(1+\dd^2)y(\dd)\left\{mN^g_{n+1}(m,b_S)\right\}_{m=-1}&\hspace{-1mm}=2\sum_{j=1}^n\hspace{-3mm}\sum_{\tiny\begin{array}{c}k=1\\k\equiv b_j(2)\end{array}}^{b_j}\hspace{-4mm}kN^g_n(b_S)|_{b_j=k}-|b_S|N^g_n(b_S)\\
(1-\dd^2)^2y(\dd)\left\{mN^g_{n+1}(m,b_S)\right\}_{m=-2}\hspace{-1mm}&=\sum_{j=1}^n\sum_{k=1\pm b_j}kN^g_n(b_S)|_{b_j=k}\\
(1-\dd^2)y(\dd)\left\{N^g_{n+1}(m,b_S)\right\}_{m=-1}\hspace{-1mm}&=(2-2g-n)N^g_n(b_S)
\label{eq:dil}
\end{align}
\end{theorem}
Although $y$ may not be a polynomial the left hand sides of (\ref{eq:str1}) - (\ref{eq:dil}) are finite sums, since large enough powers of a discrete derivative vanish on the quasi-polynomial $N^g_{n+1}$.  See Section~\ref{sec:EO}.
\begin{cor}   \label{th:sympl}
The symplectic invariants of the curve $x=z+1/z$, $y=y(z)$ satisfy
\[F^g=\frac{1}{2-2g}(1-\dd^2)y(\dd)\left\{N^g_1(m)\right\}_{m=-1}\]
for $g\geq 2$.
\end{cor}

The definition of the Eynard-Orantin invariants is given in Section~\ref{sec:EO}.  The proofs of Theorems~\ref{th:main} and \ref{th:inter} are in Section~\ref{sec:proofs} and the proof of Theorems~\ref{th:string0} and \ref{th:string} is in Section~\ref{sec:proofs}.  Section~\ref{sec:example} contains examples.

\section{Eynard-Orantin invariants.}  \label{sec:EO}

For every $(g,n)\in\bz^2$ with $g\geq 0$ and $n>0$ the Eynard-Orantin invariant of a plane curve $C$ is  a multilinear differential $\omega^g_n(p_1,...,p_n)$, i.e. a tensor product of meromorphic 1-forms on the product $C^n$, where $p_i\in C$.  When $2g-2+n>0$, $\omega^g_n(p_1,...,p_n)$ is defined recursively in terms of local  information around the poles of $\omega^{g'}_{n'}(p_1,...,p_n)$ for $2g'+2-n'<2g-2+n$.  Equivalently, the $\omega^{g'}_{n'}(p_1,...,p_n)$ are used as kernels on the Riemann surface to integrate against.  This is a familiar idea, the main example being the Cauchy kernel which gives the derivative of a function in terms of the bilinear differential $dwdz/(w-z)^2$ as follows
\[ f'(z)dz=\res{w=z}\frac{f(w)dwdz}{(w-z)^2}=-\sum_{\alpha}\res{w=\alpha}\frac{f(w)dwdz}{(w-z)^2}\]
where the sum is over all poles $\alpha$ of $f(w)$.  

The Cauchy kernel generalises to a bilinear differential $B(w,z)$ on any Riemann surface $C$ given by the meromorphic differential $\eta_w(z)dz$ unique up to scale which has a double pole at $w\in C$ and all $A$-periods vanishing.   The scale factor can be chosen so that $\eta_w(z)dz$ varies holomorphically in $w$, and transforms as a 1-form in $w$ and hence it is naturally expressed as the unique bilinear differential on $C$ 
\[ B(w,z)=\eta_w(z)dwdz,\quad \oint_{A_i}B=0,\quad B(w,z)\sim\frac{dwdz}{(w-z)^2} {\rm\  near\ }w=z.\]  
It is symmetric in $w$ and $z$.  We will call $B(w,z)$ the {\em Bergmann Kernel}, following \cite{EOrInv}.  It is called the fundamental normalised differential of the second kind on $C$ in \cite{FayThe}.  Recall that a differential is {\em normalised} if its $A$-periods vanish and it is of the {\em second kind} if its residues vanish.  It is used to express a normalised differential of the second kind in terms of local  information around its poles. 

For $2g-2+n>0$, the poles of $\omega^g_n(p_1,...,p_n)$ occur at the branch points of $x$, and they are of order $6g-4+2n$.  Since each branch point $\alpha$ of $x$ is simple, for any point $p\in C$ close to $\alpha$ there is a unique point $\hat{p}\neq p$ close to $\alpha$ such that $x(\hat{p})=x(p)$.  The recursive definition of $\omega^g_n(p_1,...,p_n)$ uses only local information around branch points of $x$ and makes use of the well-defined map $p\mapsto\hat{p}$ there. The invariants are defined as follows.
\begin{align}
\omega^0_1&=ydx\nonumber\\
 \label{eq:berg}
\omega^0_2&=B(w,z)
\end{align}
For $2g-2+n>0$,
\begin{equation}  \label{eq:EOrec}
\omega^g_{n+1}(z_0,z_S)=\sum_{\alpha}\hspace{-2mm}\res{z=\alpha}K(z_0,z)\hspace{-.5mm}\biggr[\omega^{g-1}_{n+2}(z,\hat{z},z_S)+\hspace{-5mm}\displaystyle\sum_{\begin{array}{c}_{g_1+g_2=g}\\_{I\sqcup J=S}\end{array}}\hspace{-5mm}
\omega^{g_1}_{|I|+1}(z,z_I)\omega^{g_2}_{|J|+1}(\hat{z},z_J)\biggr]
\end{equation}
where the sum is over branch points $\alpha$ of $x$, $S=\{1,...,n\}$, $I$ and $J$ are non-empty and 
\[\displaystyle K(z_0,z)=\frac{-\int^z_{\hat{z}}B(z_0,z')}{2(y(z)-y(\hat{z}))dx(z)}\] is well-defined in the vicinity of each branch point of $x$.   Note that the quotient of a differential by the differential $dx(z)$ is a meromorphic function.  The recursion (\ref{eq:EOrec}) depends only on the meromorphic differential $ydx$ and the map $p\mapsto\hat{p}$ around branch points of $x$.  

\subsection{Rational curves with two branch points.}  \label{sec:ratcurve}
The simplest example of a plane curve with non-trivial Eynard-Orantin invariants is the rational curve $y^2 = x$ where the meromorphic function $x$ defines a two-to-one branched cover with a single branch point.  It is known as the Airy curve since the Eynard-Orantin invariants reproduce KontsevichÕs generating function \cite{KonInt} for intersection numbers on the moduli space.   In this paper we study rational curves such that $x$ defines a two-to-one branched cover with two branch points.
\begin{lemma}
If $x$ is a two-to-one rational map on $\bp^1$ with two branch points then we can parametrise the domain so that $x=a+b(z+1/z)$.
\end{lemma}
\begin{proof}
Using a conformal map we can arrange that the two branch points of $x$ are $z=\pm1$.  The conformal map $z\mapsto (z+\lambda)/(\lambda z+1)$ which fixes $z=\pm1$ can be used to further arrange that $x(\infty)=\infty$.  Since $x(z)-x(1)$ has a double root at $z=1$ we have
\[ x(z)-x(1)=\frac{(z-1)^2}{q_2z^2+q_1z+q_0},\quad x(\infty)=\infty\Rightarrow q_2=0,\quad x'(-1)=0\Rightarrow q_0=0\]
so put $a=x(1)-2/q_1$ and $b=1/q_1$, and the result follows.
\end{proof}

Thus $x(z)=a+b(z+1/z)$ and $y(z)$ is any analytic function defined on a domain in $\bc$ containing $z=\pm1$ and satisfying $y'(\pm1)\neq 0$.  When $y$ is not polynomial, for example $y$ is rational or transcendental, we expand it as a series of polynomials in the following non-standard way.  Given such $y(z)$, define the partial sums
\[ y^{(N)}(z)=\sum_{k=0}^N(a_k+zb_k)(1-z^2)^k\]
to agree with $y(z)$ at $z=\pm1$ up to the $N$th derivatives.  One can achieve this by expressing $y(z)=y_+(z)+y_-(z)$ where $y_{\pm}(z)=1/2(y(z)\pm y(-z))$ and define $y^{(N)}_{+}(z)=\sum_{k=0}^Na_k(1-z^2)^k$ where $a_k$ are determined by the property 
\[\frac{d^k y_+}{dz^k}(\pm 1)=\frac{d^k y^{(N)}_{+}}{dz^k}(\pm 1),\quad k=0,...,N.\]
Similarly define $b_k$ from $y_-(z)/z$.

The partial sums $\{y^{(N)}(z)\}$ do not necessarily converge to $y(z)$.  For example, 
\[ y(z)=\ln{z}\sim\sum\frac{(1-z^2)^k}{-2k}\]
is a divergent asymptotic expansion for $\ln(z)$ at $z=0$ in the region $Re(z^2)>0$.  

The partial sums $y^{(N)}(z)$ are used in the recursions defining $\omega^g_n$ in place of $y(z)$ since they contain the same local information around $z=\pm 1$ up to order $N$.  More precisely, to define $\omega^g_n$ for the curve $(x(z),y(z))$ it is sufficient to use $(x(z),y^{(N)}(z))$ for any $N\geq 6g-6+2n$.

There are various ways to express a transcendental function as a limit of rational functions.  The main benefit of the approach used here is the fact that the expressions which appear in the string and dilaton equations $x(z)^my(z)\omega_n^g$ have poles only at $z=\pm 1$ and 0 and $\infty$, allowing one to translate properties of $\omega_n^g$ near $z=\pm 1$ to properties of $\omega_n^g$ near $z=\infty$, which is encoded by $N^{g}_n$.

Theorem~\ref{th:string} involves the expression $y(\dd)$ where $\dd$ is defined in the introduction, and in particular $I-\dd$ is a discrete derivative.  It is an easy fact (proved by induction) that for any degree $d$ polynomial $p(n)$, high enough discrete derivatives vanish: $(1-\dd)^kp(n)\equiv 0$ for $k>d$.  Similarly, $1-\dd^2$ is a discrete derivative, and for a parity dependent quasi-polynomial $p(n)$ (so $p(n)=p_+(n)$ for $n$ even and $p(n)=p_-(n)$ for $n$ odd, where $p_{\pm}(n)$ are polynomials) 
\[ (1-\dd^2)^kp(n)\equiv 0,\quad k{\rm\ sufficiently\ large}\] 
in fact $k>$ maximum degree of $p_{\pm}(n)$.  To make sense of (\ref{eq:str1}) - (\ref{eq:dil}), one must replace $y(\dd)$ with $y^{(N)}(D)$ for large enough $N$ so that the left hand sides of (\ref{eq:str1}) - (\ref{eq:dil}) have only finitely many terms.  This procedure is well-defined, since the vanishing of discrete derivatives ensures that the left hand sides of (\ref{eq:str1}) - (\ref{eq:dil}) are independent of the choice of $N$ when it is large enough.

\section{Proofs}   \label{sec:proofs}

\begin{proof}[Proof of Theorem~\ref{th:main}]

Theorem~\ref{th:main} reflects three main properties of the multilinear differential $\omega_n^g(z_1,...,z_n)$ proven in \cite{EOrInv}---it is meromorphic, with poles at $z_i=\pm1$ of order $6g-4+2n=:2d+2$ and residue 0, and possesses symmetry under $z_i\mapsto 1/z_i$.  

Since all residues of $\omega_n^g$ vanish, the integral 
\[\f^g_n(z_1,...,z_n)=\int_0^{z_1}...\int_0^{z_n}\omega_n^g(z'_1,...,z'_n)\]
is a well-defined meromorphic function that vanishes when any $z_i=0$ and has poles of order $2d+1$ at $z_i=\pm 1$.  Write this rational function as 
$$\f^g_n(z_1,...,z_n)=\frac{\sum_{0<k_i<4d+2}p_{\textbf{k}}z_1^{k_1}\dots z_n^{k_n}}{\prod_{i=1}^n (1-z_i^2)^{2d+1}}$$ 
where the $p_{\textbf{k}}=p_{k_1,\dots,k_n}\in \mathbb{C}$ and the degree of the numerator is small enough to avoid a pole at infinity. 

The Taylor expansion
\[\frac{1}{(1-z^2)^{2d+1}}= \frac{1}{(2d)!}\frac{d^{2d}}{d(z^2)^{2d}}\sum_{m=0}^{\infty}z^{2m}=\sum_{m=0}^\infty\binom{m+2d}{2d}z^{2m}\]
has quasi-polynomial coefficients, meaning that the coefficients of $z^b$ are described by two polynomials in $b$---when $b$ is odd the coefficient of $z^b$ is the zero polynomial and when $b$ is even the coefficient of $z^b$ is a degree $2d$ polynomial in $b$.  More generally, the Taylor expansion of $\f^g_n(z_1,...,z_n)$ about $z_i=0$ has quasi-polynomial coefficients, depending on parity.  When $n=1$,
\[\frac{\sum p_kz^k}{(1-z^2)^{2d+1}}=  \sum_{k,m}p_k\binom{m+2d}{2d}z^{2m+k}=\sum_{b>0} N^g_1(b)z^b.\]
The coefficient of $z^b$ consists of all terms where $2m+k=b$, hence the odd part of $p(z)=\sum p_kz^k$ gives rise to a degree $2d$ polynomial representing $N^g_1(b)$ when $b$ is odd, and the even part of $p(z)$ gives rise to a degree $2d$ polynomial representing $N^g_1(b)$ when $b$ is even.  Similarly,
\begin{align*}
\f_n^g(z_1,...,z_n)&=\sum_{k_1,...,k_n=0}^{4d+2}p_{\textbf{k}}z_1^{k_1}\dots z_n^{k_n}\prod_{i=1}^n \sum_{m_i=0}^\infty\binom{m_i+2d}{2d}z_i^{2m_i}\\
&=\sum_{k_1,...,k_n=0}^{4d+2}\sum_{m_i\geq 0} p_{{\textbf{k}}}\prod_{i=1}^n\binom{m_i+2d}{2d}z_i^{2m_i+k_i}\\
&=:\ \ \sum_{b_i>0}^\infty N^g_n(b_1,\dots,b_n)z_1^{b_1}\dots z_n^{b_n}
\end{align*}
expresses $N^g_n(b_1,\dots,b_n)$ as the sum over the terms with $2m_i+k_i=b_i$ which is a quasi-polynomial depending on the parity of the $b_i$.   By symmetry of the $z_i$, $N^g_n$ does not depend on which $b_i$ are odd but only how many.  Hence we write   
\[ N^g_n(b_1,...,b_n)=N^g_{n,k}(b_1,...,b_n),\quad k={\rm number\ of\ odd\ }b_i.\]
Each binomial coefficient and hence each polynomial $N^g_{n,k}(b_1,\dots,b_n)$ is a polynomial of degree $2d$ in each $b_i$.  The stronger fact that they have {\em homogeneous} degree $2d$ in the $b_i$ is a consequence of Theorem~\ref{th:inter}.

It remains to show that $N^g_n$ is a quasi-polynomial in the $b_i^2$.  Equivalently, we will show that $b_1...b_nN^g_n(b_1,\dots,b_n)$ is odd in each $b_i$ using symmetries of
\[\omega^g_n=\sum_{b_i>0}^\infty b_1...b_nN^g_n(b_1,\dots,b_n)z_1^{b_1-1}\dots z_n^{b_n-1}dz_1...dz_n.\]
\begin{lemma}
A meromorphic 1-form on $\bp^1$ with poles at $z=\pm 1$ has the following related expansions around $z=0$
\begin{equation}  \label{eq:recneg}
\omega(z)=\sum_{n=1}^{\infty}p(n)z^{n-1}dz\quad\Leftrightarrow\quad\omega\left(\frac{1}{z}\right)=\sum_{n=1}^{\infty}p(-n)z^{n-1}dz
\end{equation}
where $p(n)$ is a quasi-polynomial depending on the parity of $n$.
\end{lemma}
\begin{proof}
We can express $\omega$ as a rational function with numerator a polynomial of degree small enough so that there are no poles at infinity.  In particular, by linearity it is enough to prove the lemma when the numerator is a monomial so
\[\omega(z)=\frac{z^kdz}{(1-z^2)^{m+1}}\quad{\rm and}\quad \omega\left(\frac{1}{z}\right)=(-1)^m\frac{z^{2m-k}dz}{(1-z^2)^{m+1}}.\]
From the expansion
\[\frac{1}{(1-z^2)^{m+1}}=\sum\binom{n+m}{m}z^{2n}\]
one gets
\[\frac{z^kdz}{(1-z^2)^{m+1}}=\sum\binom{n+m}{m}z^{k+2n}dz=\sum p_k(b)z^{b-1}dz\]
where 
\[
p_k(b)=\left\{\begin{array}{cl} 0,&b\equiv k ({\rm mod\ }2)\\\binom{(b-k-1)/2+m}{m}, &b\not\equiv k ({\rm mod\ }2)\end{array}\right..\]
Also, 
\[p_{2m-k}(b)=\binom{(b-2m+k-1)/2+m}{m}=\binom{(b+k-1)/2}{m}=(-1)^mp_k(-b)\]
(for $b\not\equiv k ({\rm mod\ }2)$ and $p_{2m-k}(b)=0$ for $b\equiv k ({\rm mod\ }2)$ so the above equation holds for all $b$.)  Hence (\ref{eq:recneg}) holds when $\omega$ has a monomial numerator and hence for all rational $\omega$ and the lemma is proven.
\end{proof}
An immediate corollary of the lemma is that if $\omega(z)=\omega(1/z)$ then the quasi-polynomial $p(n)$ is even in $n$ and if $\omega(z)=-\omega(1/z)$ then $p(n)$ is odd in $n$.  A consequence of a more general result in \cite{EOrInv} is the symmetry
\[\omega_n^g(z_1,\dots,z_n)=-\omega^g_n(1/z_1,\dots,z_n)\] 
and similarly for each variable $z_i$.  Hence $b_1...b_nN^g_n(b_1,\dots,b_n)$ is an odd quasi-polynomial in each $b_i$ as required.
\end{proof}

\noindent {\em Remark.}  The expansion (\ref{eq:expand}) of $\omega^g_n(z_1,...,z_n)$ around $(z_1,...,z_n)=(0,...,0)$ defines $N^g_n(b_1,...,b_n)$ only for $(b_1,...,b_n)\in\bz_+^n$.  One can make sense of $b_i=0$ using the polynomial representation $N^g_{n,k}(b_1,...,b_n)$ of $N^g_n(b_1,...,b_n)$ for $k=$ number of odd $b_i$.  In terms of $\omega^g_n$ one has the following
\[N^g_1(0)=\int^0_{\infty}\omega^g_1(z)\]
and more generally, $b_1...b_kN^g_n(b_1,...,b_k,0,...,0)$ is the coefficient of $z_1^{b_1-1}...,z_k^{b_k-1}$ in the expansion around $(z_1,...,z_k)=(0,...,0)$ of $\int_{z_{k+1}=\infty}^0...\int_{z_n=\infty}^0\omega^g_n(z_1,...,z_n)$.

\begin{proof}[Proof of Theorem~\ref{th:inter}]

The proof uses the behaviour of $\omega^g_n$ near the branch points $z_i=\pm 1$.  Express $\omega^g_n$ as a rational function
$$\omega^g_n=\frac{\sum^{4d+2}_{k_1,\dots,k_n=0}c_{\textbf{k}}z_1^{k_1}\dots z_n^{k_n}}{\prod_{i=1}^n (1-z_i^2)^{2d+2}}dz_1...dz_n$$
for $c_{\textbf{k}}\in \mathbb{C}$, and $d=3g-3+n$.  Consider the change of variables $z_i=\epsilon_i+sx_i$ where $\epsilon_i\in \{\pm 1\}$, $s\in\mathbb{R}$ is small and $x_i$ is a local coordinate on the spectral curve.  The asymptotic behaviour of $\omega^g_n$ near $z_i=\pm 1$ corresponds to $s\rightarrow 0$ for all combinations of the $\epsilon_i$.  This change gives:
$$\omega^g_n=\frac{\sum^{4d+2}_{k_1,\dots,k_n=0}c_{\textbf{k}}(\epsilon_1+sx_1)^{k_1}\dots (\epsilon_n+sx_n)^{k_n}}{s^{(2d+2)n}\prod_{i=1}^nx_i^{2d+2} (2\epsilon_i+sx_i)^{2d+2}}{s^n\Pi^n_{i=1}}dx_i,$$
and we must find a minimal $q=q(\epsilon_i)\in \{0,1,\dots 4d+2\}$ so that the coefficient of $s^q$ in the numerator is the first non vanishing.  For example, if $q>0$ this tells us 
$$\sum^{4d+2}_{k_1,\dots,k_n=0}c_{\textbf{k}}\prod_{i=1}^n \epsilon_i^{k_i}=0,$$  
and if $q>1$, then the coefficient of $s^1=0$.  That is: 
$$\sum^{4d+2}_{k_1,\dots,k_n=0}c_{\textbf{k}}\prod_{i=1}^n \epsilon_i^{k_i}\left(\binom{k_1}{1}\frac{x_1}{\epsilon_1}+\dots +\binom{k_n}{1}\frac{x_n}{\epsilon_n}\right)=0$$ 
and by equating coefficients of $x_j$, for  $1\leq j \leq n$: 
$$\sum^{4d+2}_{k_1,\dots,k_n=0}c_{\textbf{k}}\prod_{i=1}^n\epsilon_i^{k_i}\frac{k_j}{\epsilon_j}=0.$$
For a general $q$ and $\alpha=(\alpha_1,\dots,\alpha_n)$
\begin{equation*}
\sum^{4d+2}_{k_1,\dots,k_n=0}c_{\textbf{k}}\binom{k_1}{\alpha_1}\dots \binom{k_n}{\alpha_n}\epsilon_1^{k_1-\alpha_1}\dots \epsilon_n^{k_n-\alpha_n}x_1^{\alpha_1}\dots x_n^{\alpha_n}=0 \indent\text{if $|\alpha|<q$.}
\end{equation*}
Thus inductively one gets
\begin{equation} \label{eq:vanishsum}
\sum^{4d+2}_{k_1,\dots,k_n=0}c_{\textbf{k}}\epsilon_1^{k_1-\alpha_1}\dots \epsilon_n^{k_n-\alpha_n}k_1^{\alpha_1}\dots k_n^{\alpha_n}=0 \indent\text{if $|\alpha|<q$}.
\end{equation}
This means that the dominant asymptotic term as $s\rightarrow 0$ will look like:
\begin{equation*}
\omega^g_n\sim \frac{1}{s^{(2d+1)n-q}\prod_{i=1}^nx_i^{2d+2}}\sum^{4d+2}_{k_1,\dots,k_n=0}\sum_{|\alpha|=q}\frac{c_{\textbf{k}}}{{2^{(2d+2)n}\alpha!}}{\Pi^n_{i=1}}{\epsilon_i^{k_i-\alpha_i}k_i^{\alpha_i}x_i^{\alpha_i} }dx_i.
\end{equation*}
In \cite{EOrAlg} it is shown that as $z_1,\dots, z_n$ tends to the branch points $\epsilon_1,\dots, \epsilon_n$, $\epsilon_j=\pm 1$
\begin{equation*}
\omega^g_n\sim \begin{cases}
{s^{6-6g-3n}[\frac{1}{2}x''(\epsilon_i)y'(\epsilon_i)]^{2-2g-n}\omega^g_n[Airy], \indent \text{for all $\epsilon_i$ the same}}\\
{\quad\text{lower order asymptotics,\quad\quad\quad\quad\quad\quad\quad for mixed $\epsilon_i$}}
\end{cases}
\end{equation*} where the Airy curve is given by $y^2=x$.  Thus $q=2d(n-1)$ if all $\epsilon_i$ are the same and $q>2d(n-1)$ for all other combinations.  

From \cite{EOrInv} there is a relationship between $\omega^g_n[Airy]$ and intersection numbers of tautological line bundles over the moduli space.
\begin{equation}
\omega^g_n[Airy](z_S)=\frac{x''(0)^{2-2g-n}}{2^{3g-3+n}}\sum_{|\beta|=d}\prod_{i=1}^n \frac{(2\beta_i+1)!}{\beta_i!}\frac{dz_i}{z_i^{2\beta_i+2}}\langle \tau_{\beta_1}...\tau_{\beta_n}\rangle,
\end{equation}
where we have used Witten's \cite{WitTwo} notation $\langle \tau_{\beta_1}...\tau_{\beta_n}\rangle=\int_{\overline{\modm}_{g,n}}c_1(L_1)^{\beta_1}...c_1(L_n)^{\beta_n}$ and $x''(0)=2$.  From this we discover that if all $\epsilon_i=1$:
\begin{eqnarray*}
&\frac{1}{\prod_{i=1}^nx_i^{2d+2}}&\sum^{4d+2}_{k_1,\dots,k_n=0}\sum_{|\alpha|=2d(n-1)}\frac{c_{\textbf{k}}}{{2^{(2d+2)n}\alpha!}}{\prod_{i=1}^nk_i^{\alpha_i}x_i^{\alpha_i}}dx_i\\
&=&\sum^{4d+2}_{k_1,\dots,k_n=0}\sum_{|\alpha|=2d(n-1)}\frac{c_{\textbf{k}}}{2^{(2d+2)n}}\prod_{i=1}^n\frac{k_i^{\alpha_i}x_i^{\alpha_i-2d-2}}{\alpha_i!}dx_i\\
&=&\frac{\left\{x''(1)y'(1)\right\}^{2-2g-n}}{2^{3g-3+n}}\sum_{|\beta|=d}\prod_{i=1}^n \frac{(2\beta_i+1)!}{\beta_i!}\frac{dx_i}{x_i^{2\beta_i+2}}\langle \tau_{\beta_1}...\tau_{\beta_n}\rangle.
\end{eqnarray*}
and so each $\alpha_i$ must be even and $0\leq \alpha_i\leq 2d$.  By equating powers of $x_i$ (that is, extracting the partition where $\alpha_i=2d-2\beta_i$) one gets the relation
\begin{equation} \label{eq:airyrel}
\sum^{4d+2}_{k_1,\dots,k_n=0}\frac{c_{\textbf{k}}}{2^{(2d+2)n}}\prod_{i=1}^n\frac{k_i^{2d-2\beta_i}}{(2d-2\beta_i)!}
=\frac{\left\{x''(1)y'(1)\right\}^{2-2g-n}}{2^{3g-3+n}}\prod_{i=1}^n \frac{(2\beta_i+1)!}{\beta_i!}\langle \tau_{\beta_1}...\tau_{\beta_n}\rangle.
\end{equation}

Similarly, when all $\epsilon_i=-1$ one gets 
\begin{align} \label{eq:airyrel2}
\sum^{4d+2}_{k_1,\dots,k_n=0}\frac{c_{\textbf{k}}}{2^{(2d+2)n}}\prod_{i=1}^n&\frac{(-1)^{k_i}k_i^{2d-2\beta_i}}{(2d-2\beta_i)!}\\
&=\frac{\left\{x''(-1)y'(-1)\right\}^{2-2g-n}}{2^{3g-3+n}}\prod_{i=1}^n \frac{(2\beta_i+1)!}{\beta_i!}\langle \tau_{\beta_1}...\tau_{\beta_n}\rangle.\nonumber
\end{align}
and when there is a mix of $\epsilon_i$'s, $q$ will be greater, introducing more vanishing:
\begin{equation} \label{eq:airyrel3}
\sum^{4d+2}_{k_1,\dots,k_n=0}\frac{c_{\textbf{k}}}{2^{(2d+2)n}}\prod_{i=1}^n\frac{\epsilon_i^{k_i}k_i^{2d-2\beta_i}}{(2d-2\beta_i)!}=0 \indent \text{for $|\beta|=d$.}
\end{equation}

Equations~(\ref{eq:airyrel}), (\ref{eq:airyrel2}) and (\ref{eq:airyrel3}) translate to an analogous equation (\ref{eq:coefflin}) for coefficients of the polynomials $N^g_{n,k}$.  To show this we now study the polynomials in terms of the coefficients $c_{\textbf{k}}$.  As in the proof of Theorem ~\ref{th:main}, the Taylor expansion of $\omega^g_n$ about $z_i=0$ can be written:
\begin{align*}
\omega^g_n&=\sum_{{k_1,...,k_n=0}}^{4d+2}\sum_{b_1,\dots,b_n=0}^{\infty} c_{{\textbf{k}}}\prod_{i=1}^n\binom{(b_i+1-k_i)/2+2d}{2d+1}z_i^{b_i-1} dz_i\\
&=\frac{2^{-(2d+1)n}}{(2d+1)!^n}\sum_{k_1,\dots,k_n=0}^{4d+2}\sum_{b_1,\dots,b_n=0}^{\infty} c_{{\textbf{k}}}\prod_{i=1}^n\left(\sum_{j=0}^{2d+1}\sigma_j(k_i)b_i^{2d+1-j}\right)z_i^{b_i-1} dz_i\\
&=\sum_{b_i>0}^\infty b_1...b_nN^g_n(b_1,\dots,b_n)z_1^{b_1-1}\dots z_n^{b_n-1}dz_1...dz_n
\end{align*}
where for $b_i$ even, respectively odd, we only sum over the odd, respectively even, $k_i$ and $\sigma_j(k_i)$ (= coefficient of $b_i^{2d+1-j}$ in $(b_i+4d+1-k_i)\dots(b_i+3-k_i)(b_i+1-k_i)$) is a degree $j$ polynomial in $k_i$.

Thus the homogeneous degree $2q$ terms of the quasi-polynomial $N^g_n$ are
\begin{equation}  \label{eq:homdeg}
(\prod_{i=1}^n b_i)N^g_n[\text{degree\ }2q]=\frac{2^{-(2d+1)n}}{(2d+1)!^n}\sum_{k_1,\dots,k_n=0}^{4d+2} \sum_{|\beta|=q}c_{{\textbf{k}}}\prod_{i=1}^n\sigma_{2d-2\beta_i}(k_i)b_i^{2\beta_i+1}
\end{equation}
where we are still summing over parity dependent \textbf{k}.
 
The equations~(\ref{eq:vanishsum}), (\ref{eq:airyrel}), (\ref{eq:airyrel2}) and (\ref{eq:airyrel3}) give identities for sums over {\em all} $k_i$ of  $c_{{\textbf{k}}}$ times monomials in $k_i$ and we wish to apply these to (\ref{eq:homdeg}) which consists of coefficients that sum over only some of the $k_i$, depending on parity.   To remedy this, we add together the different polynomials representing $N^g_n$ for every possible parity.  This removes the restriction on the \textbf{k} summand.   

For $\{i_1,\dots,i_k\}\subset\{1,...,n\}$, define $v^{\{i_1,\dots,i_k\}}_{\beta}$ to be the coefficient of $b_1^{2\beta_1}\dots b_n^{2\beta_n}$ in $N^g_{n,k}$ with $b_{i_1},\dots , b_{i_k}$ odd.  For $\epsilon=(\epsilon_1,...,\epsilon_n) \in \{\pm 1\}^n$ define $\epsilon^{I}=\prod_{k\in I} \epsilon_k$.   Since $\sigma_j(k_i)$ is a polynomial of degree $j$ in $k_i$, if $|\beta|=q>d$, then the homogeneous degree in $k_1,k_2,...,k_n$ of the product of $\sigma_{2d-2\beta_i}(k_i)$ is small enough that (\ref{eq:vanishsum}) implies the vanishing of each of the coefficients.  In other words we have shown that 
\begin{equation}  \label{eq:coefflin0}
\sum_{I\subset \{1,\dots, n\}}\epsilon^{I}v^I_\beta=0,\quad |\beta|>d.
\end{equation}

If $|\beta|=q=d$, then the only non zero sums resulting from the $\sigma_{2d-2\beta_i}(k_i)$ are the top powers of the $k_i$, that is $\prod_{i=1}^n(-k_i)^{2d-2\beta_i}$, since any component with a smaller power of $k_i$'s will vanish when summed again by (\ref{eq:vanishsum}).  There will be $\prod_{i=1}^n \binom{2d+1}{2d-2\beta_i}$ of these.  This leaves
\begin{eqnarray*}
&&\sum_{\text{parities}} N^g_{n}[\text{degree\ }2d]\\
&&=\frac{2^{-(2d+1)n}}{(2d+1)!^n}\sum_{k_1,\dots,k_n=0}^{4d+2}  \sum_{|\beta|=d}c_{{\textbf{k}}}\prod_{i=1}^n\binom{2d+1}{2d-2\beta_i}(-1)^{2d-2\beta_i}k_i^{2d-2\beta_i}b_i^{2\beta_i}\\
&&=\frac{1}{2^{(2d+1)n}}\sum_{k_1,\dots,k_n=0}^{4d+2}  \sum_{|\beta|=d}c_{{\textbf{k}}}\prod_{i=1}^n\frac{k_i^{2d-2\beta_i}}{(2d-2\beta_i)!}\frac{b_i^{2\beta_i}}{(2\beta_i+1)!}
\end{eqnarray*}
thus
\begin{equation*}
\sum_{I\subset \{1,\dots, n\}}v^I_\beta=\frac{1}{2^{(2d+1)n}}\sum_{k_1,\dots,k_n=0}^{4d+2} c_{{\textbf{k}}}\prod_{i=1}^n\frac{k_i^{2d-2\beta_i}}{(2d-2\beta_i)!(2\beta_i+1)!}
\end{equation*}
where we are again extending to the full sum over all $\textbf{k}$'s.
More generally
\begin{equation*}
\sum_{I\subset \{1,\dots, n\}}\epsilon^Iv^I_\beta=\frac{1}{2^{(2d+1)n}}\sum_{k_1,\dots,k_n=0}^{4d+2} c_{{\textbf{k}}}\prod_{i=1}^n\frac{\epsilon_i^{k_i+1}k_i^{2d-2\beta_i}}{(2d-2\beta_i)!(2\beta_i+1)!}
\end{equation*}

Now we can see that these are (up to simple combinatorial factors) the asymptotic formulas (\ref{eq:airyrel}), (\ref{eq:airyrel2}) and (\ref{eq:airyrel3})!   We now have expressions for these in terms of intersection numbers.

\begin{equation}   \label{eq:coefflin}
\sum_{I\subset \{1,...,n\}}\epsilon^Iv_\beta^I=\begin{cases}
{\frac{\left\{x''(1)y'(1)\right\}^{2-2g-n}}{2^{3g-3}\beta!}\langle \tau_{\beta_1}...\tau_{\beta_n}\rangle, \indent \text{$\epsilon_i=1$ for all $i$}}\\
{\frac{(-1)^n\left\{x''(-1)y'(-1)\right\}^{2-2g-n}}{2^{3g-3}\beta!}\langle \tau_{\beta_1}...\tau_{\beta_n}\rangle,   \ \text{$\epsilon_i=-1$ for all $i$,}}\\
{0 \indent \text{otherwise.}}
\end{cases}\hspace{-1mm}|\beta|=d.
\end{equation}

By varying $\epsilon=(\epsilon_1,...,\epsilon_n)\in \{\pm 1\}^n$, (\ref{eq:coefflin0}) and (\ref{eq:coefflin}) are sets of $2^n$ equations with $2^n$ unknowns $v^I_\beta$ that can be uniquely solved.  When $n=1$, the two equations use the matrix $J=\left(\begin{array}{cc}1&1\\1&-1\end{array}\right)$, and more generally the $2^n$ equations use the $2^n\times 2^n$ matrix given by the tensor product $M=J^{\otimes n}$---equivalently $M$ is the linear map induced by $J$ on the tensor product $(\bc^2)^{\otimes n}$.   The matrix $M$ is orthogonal (up to a $2^n$ scaling factor) 
$$MM^T=\left(JJ^T\right)^{\otimes n}=\left(2I\right)^{\otimes n}=2^nI.$$
Assemble $v^I_\beta$ into a $2^n$-vector $v_{\beta}$.
Thus (\ref{eq:coefflin0}) becomes 
\begin{equation*}  \label{eq:coefflin0'}
Mv_{\beta}=0,\quad |\beta|>d.
\tag{\ref{eq:coefflin0}$'$}
\end{equation*} 
and (\ref{eq:coefflin}) becomes $Mv_{\beta}=$ the right hand side of (\ref{eq:coefflin}) or more explicitly 
\begin{equation*}  \label{eq:coefflin'} 
Mv_{\beta}=\begin{pmatrix} 
 \left\{x''(1)y'(1)\right\}^{2-2g-n}\\ 
  0\\
\vdots\\
  (-1)^n\left\{x''(-1)y'(-1)\right\}^{2-2g-n} 
\end{pmatrix} \frac{\langle \tau_{\beta_1}...\tau_{\beta_n}\rangle}{2^{3g-3}\beta!},\quad |\beta|=d
\tag{\ref{eq:coefflin}$'$}
\end{equation*} 
where the first and last rows of $M$ are the $2^n$-vectors $e_0=(1,1,..)$ and $e_1=\{(-1)^I\}$ corresponding to $\epsilon_i=1$ for all $i$, respectively  $\epsilon_i=-1$ for all $i$.  In particular, $M$ is invertible so equations~(\ref{eq:coefflin0'}) and (\ref{eq:coefflin'}) have the unique solutions $v_{\beta}=0$ when $|\beta|>d$ and when $|\beta|=d$, $v_{\beta}$ lies in the plane spanned by $e_0$ and $e_1$.  Explicitly
\begin{align*}
v_{\beta}=&\frac{\langle \tau_{\beta_1}...\tau_{\beta_n}\rangle}{2^{3g-3+n}\beta!}\left(\left\{x''(1)y'(1)\right\}^{2-2g-n}e_0+(-1)^n\left\{x''(-1)y'(-1)\right\}^{2-2g-n}e_1\right)\\
=&\frac{\langle \tau_{\beta_1}...\tau_{\beta_n}\rangle}{x''(1)^{2g-2+n}2^{3g-3+n}\beta!}\left(y'(1)^{2-2g-n}e_0+y'(-1)^{2-2g-n}e_1\right)
\end{align*}
where we have used $x''(-1)=-x''(1)$.  Equivalently
$$v^I_\beta=\frac{y'(1)^{2-2g-n}+(-1)^{|I|}y'(-1)^{2-2g-n}}{x''(1)^{2g-2+n}2^{3g-3+n}\beta!}\langle \tau_{\beta_1}...\tau_{\beta_n}\rangle.$$
In particular, the maximal homogeneous degree of each polynomial $N^g_{n,k}(b_1,...,b_n)$ is $2d=6g-6+2n$ and the top coefficents are given in terms of intersection numbers as claimed.
\end{proof}

\section{String and dilaton equations}  \label{sec:string}
The Eynard-Orantin invariants $\omega^g_n$ satisfy the following {\em string equations} \cite{EOrInv}.
\begin{equation}  \label{eq:string}
\sum_{\alpha} \res{z=\alpha} x^my\omega^g_{n+1}(z,z_S)=-\sum_{j=1}^ndz_j\frac{\partial}{\partial z_j}\big(\frac{x^m(z_j)\omega^g_n(z_S)}{dx(z_j)}\big)
\end{equation}
for $m=0,1$ or 2, $\alpha$ the poles of $ydx$ and $z_S=\{z_1,\dots,z_n\}$.  Note that the $m=2$ case only works for $y$ being a sum of monomials, since the proof of the above equation in \cite{EOrInv} requires $\frac{x^m(z)}{dx(z)}$ to not have a pole at any of the poles of $ydx(z)$.

\subsection{Proof of Theorem~\ref{th:string}}
If $y(z)$ is a polynomial, represent it as a finite sum $y(z)=y_0+y_1z+y_2z^2+...$ .  More generally, as described in Section~\ref{sec:EO}, we can approximate any analytic $y(z)$ by a polynomial $y^{(N)}(z)$ which agrees with $y(z)$ at $z=\pm1$ up to the $N$th derivatives.  Express the polynomial as a finite sum $y^{(N)}(z)=y_0+y_1z+y_2z^2+...$ .  Note that when $y(z)$ is not a polynomial then the finite sum $y_0+y_1z+y_2z^2+...$ is not a partial sum for a Taylor series of $y(z)$ at $z=0$.  In the following, we choose $N\geq 6g-6+2(n+1)$ so that $y(z)$ can be replaced in the left hand side of (\ref{eq:string}) by $y^{(N)}(z)$ in the residue calculations.\\
\\
\fbox{$m=0$}
\begin{eqnarray*}
&&\sum_{\alpha=\pm1}\res{z=\alpha}y(z)\omega^g_{n+1}(z,z_S) = -\sum_{\alpha=0,\infty}\res{z=\alpha}(y_1z+y_2z^2+\dots )\omega^g_{n+1}(z,z_S) \\
&& = -\res{z=\infty}(y_1z+y_2z^2+\dots )\omega^g_{n+1}(z,z_S) \\
&& = -\res{z=0}(\frac{y_1}{z}+\frac{y_2}{z^2}+\dots )(-\omega^g_{n+1}(z,z_S)) \\
\end{eqnarray*}
The coefficient of $\prod_{i=1}^n b_iz_i^{b_i-1}$ in the expansion about zero is 
$$y_1N^g_{n+1}(b_1,\dots, b_n,1)+2y_2N^g_{n+1}(b_1,\dots,b_n,2)+\dots = \sum_{k=1}^{\infty}ky_kN^g_{n+1}(b_1,\dots,b_n,k)$$
which can be expressed as $\dd y(\dd)\left\{mN^g_{n+1}(m,b_S)\right\}_{m=-1}$ as required.

The right hand side of (\ref{eq:string}), as shown in \cite{NorStr}, expands as 
\begin{equation}
\sum_{j=1}^n\frac{\partial}{\partial z_j}[(z_j^2+z_j^4+z_j^6+\dots)\omega^g_n(z_S)]
\end{equation}
and the coefficient of $\prod_{i=1}^n b_iz_i^{b_i-1}$ in the expansion about zero is given by 
$$\sum_{j=1}^n\hspace{-3mm}\sum_{\tiny\begin{array}{c}k=1\\k\not\equiv b_j(2)\end{array}}^{b_j}\hspace{-3mm}kN^g_n(b_1,\dots,b_n)|_{b_j=k}$$

\fbox{$m=1$}
\begin{eqnarray*}
&& \sum_{\alpha=\pm1}\res{z=\alpha}x(z)y(z)\omega^g_{n+1}(z,z_S)\\
&&=-\sum_{\alpha=0,\infty}\res{z=\alpha}(y_1+y_2z+(y_1+y_3)z^2+(y_2+y_4)z^3+\dots)\omega^g_{n+1}(z,z_S)\\
&&=-\res{z=\infty}(y_1+y_2z+(y_1+y_3)z^2+(y_2+y_4)z^3+\dots)\omega^g_{n+1}(z,z_S)\\
&&=\res{z=0}(y_1+\frac{y_2}{z}+\frac{y_1+y_3}{z^2}+\frac{y_2+y_4}{z^3}+\dots)\omega^g_{n+1}(z,z_S).
\end{eqnarray*}
The coefficient of $\prod_{i=1}^n b_iz_i^{b_i-1}$ in the expansion about zero is 
$$y_2N^g_{n+1}(1,b_1\dots,b_n)+\sum_{k=2}^{\infty}k(y_{k-1}+y_{k+1})N^g_{n+1}(k,b_1,\dots,b_n)$$
and if we add $-y_0N^g_{n+1}(-1,b_1,\dots,b_n)+y_0N^g_{n+1}(1,b_1\dots,b_n)=0$ this can be expressed as $(1+\dd^2)y(\dd)\left\{mN^g_{n+1}(m,b_S)\right\}_{m=-1}$.

The right hand side of (\ref{eq:string}) can be expanded as 
\begin{equation}
\sum_{j=1}^n\frac{\partial}{\partial z_j}[(z_j+2z_j^3+2z_j^5+\dots)\omega^g_n(z_S)]
\end{equation}
and the coefficient of $\prod_{i=1}^n b_iz_i^{b_i-1}$ in the expansion about zero is
\[2\sum_{j=1}^n\hspace{-3mm}\sum_{\tiny\begin{array}{c}k=1\\k\equiv b_j(2)\end{array}}^{b_j}\hspace{-3mm}kN^g_n(b_1,\dots,b_n)|_{b_j=k}+b_jN^g_n(b_1,\dots, b_n).\]
\fbox{$m=2$}
\begin{eqnarray*}
&& \sum_{\alpha=\pm1}\res{z=\alpha}x^2(z)y(z)\omega^g_{n+1}(z,z_S)\\
&=&-\sum_{\alpha=0,\infty}\res{z=\alpha}(\frac{y_1}{z}+y_2+(y_3+2y_1)z+\dots )\omega^g_{n+1}(z,z_S)\\
&=&-\res{z=0}\frac{y_1}{z}\omega^g_{n+1}(z,z_S)-\res{z=\infty}(y_2+(y_3+2y_1)z+\dots )\omega^g_{n+1}(z,z_S)\\
&=&\res{z=0}\left(-\frac{y_1}{z}+\frac{y_3+2y_1}{z}+\frac{y_4+2y_2}{z^2}+\frac{y_5+2y_3+y_1}{z^3}+\dots \right)\omega^g_{n+1}(z,z_S)\\
\end{eqnarray*}
and the right hand side expands as:
\begin{equation}
\sum_{j=1}^n\frac{\partial}{\partial z_j}[(1+3z_j^2+4z_j^4+4z_j^6+\dots)\omega^g_n(z_S)].
\end{equation}
In this case it is convenient to subtract four times the $m=0$ case (equivalently we put $x^2-4$ in place of $x^2$ in (\ref{eq:string}).)  Once again, collecting the coefficient of $\prod_{i=1}^n b_iz_i^{b_i-1}$ in the expansion about zero of both sides gives the result.

\subsection{Dilaton}

The Eynard-Orantin invariants also satisfy the {\em dilaton equation.}
\begin{equation}  \label{eq:dilaton}
\sum_{\alpha=\pm 1}\res{z=\alpha}\Phi(z)\omega_{n+1}^g(z,z_S)=(2g-2+n)\omega_{n}^g(z_S)
\end{equation}
where $d\Phi=ydx$.  The function $\Phi$ is well-defined up to a constant in a neighbourhood of each branch point and the left hand side of (\ref{eq:dilaton}) is independent of the choice of constant.
\begin{proof}[Proof of equation~(\ref{eq:dil})]
Starting from the Dilaton equation proven in \cite{EOrInv}, we let $\Phi(z)$ be an arbitrary anti derivative of $ydx$. That is, $d\Phi(z) = ydx(z)$.  Then
\begin{equation}
\sum_{\alpha=\pm1} \res{z=\alpha} \Phi(z_{n+1})\omega^g_{n+1}(z,z_S)=(2g-2+n)\omega^g_n(z_S)
\end{equation}
We can manipulate this, using integration by parts to rewrite the left hand side as: 
\begin{eqnarray*}
&&\sum_{\alpha=\pm1} \res{z=\alpha} \Phi(z_{n+1})\omega^g_{n+1}(z,z_S)=-\sum_{\alpha=\pm1}\res{z=\alpha}d\Phi(z)\int_0^z \omega^g_{n+1}(z',z_S)\\
&& =-\sum_{\alpha=\pm1}\res{z=\alpha}(\frac{-y_1}{z}-y_2+(y_1-y_3)z+(y_2-y_4)z^2+\dots )dz\int_0^z \omega^g_{n+1}(z',z_S)\\
&& =\res{z=\infty}(\frac{-y_1}{z}-y_2+(y_1-y_3)z+(y_2-y_4)z^2+\dots )dz\int_0^z \omega^g_{n+1}(z',z_S)\\
&&=-\res{z=\infty}(-y_2z+\frac{y_1-y_3}{2}z^2+\dots )\omega^g_{n+1}(z,z_S)-\res{z=\infty}\frac{y_1dz}{z}\int_0^z \omega^g_{n+1}(z',z_S)\\
&&=\res{z=0}(-\frac{y_2}{z}+\frac{y_1-y_3}{2z^2}+\dots )\omega^g_{n+1}(z,z_S)-\res{z=\infty}\frac{y_1}{z}\int_0^z \omega^g_{n+1}(z',z_S)
\end{eqnarray*}
where lines two to three uses the sum of residues of a meromorphic function is zero, and lines three to four uses integration by parts.
The last residue about the simple pole $z=\infty$ can be computed, since $\int_0^z \omega^g_{n+1}(z',z_S)$ is analytic there.  Thus 
\[\res{z=\infty}\frac{y_1dz}{z}\int_0^z \omega^g_{n+1}(z',z_S)=-y_1\int_0^\infty \omega^g_{n+1}(z',z_S)\] 
and as in the remark after the proof of Theorem~\ref{th:main} we can extract the coefficient of $\prod_{i=1}^n b_iz_i^{b_i-1}$ in the expansion about zero to get $y_1N^g_{n+1}(0,b_1,\dots,b_n)$.  Hence the total coefficient of $\prod_{i=1}^n b_iz_i^{b_i-1}$ in the expansion about zero is 
$$-y_1N^g_{n+1}(0,b_1,\dots,b_n)-y_2N^g_{n+1}(1,b_1\dots,b_n)+\sum_{k=2}^{\infty}(y_{k-1}-y_{k+1})N^g_{n+1}(k,b_1,\dots, b_n)$$
and if we add $y_0N^g_{n+1}(1,b_1,\dots,b_n)-y_0N^g_{n+1}(-1,b_1\dots,b_n)=0$ this can be expressed as $-(1-\dd^2)y(\dd)\left\{N^g_{n+1}(m,b_S)\right\}_{m=-1}$ as required.
\end{proof}
\noindent {\em Remark.} It was necessary in the proof of Theorem~\ref{th:string} that the $b_i>0$.   The equations immediately extend to allow all $b_i$, since the left hand side and right hand side are polynomials that agree at infinitely many values in each variable hence they coincide.  For example, when $x=z+1/z$ and $y=z$ the dilaton equation yields
\[N^g_{n+1,k}(2,b_1,...,b_n)-N^g_{n+1,k}(0,b_1,...,b_n)=(2g-2+n)N^g_{n,k}(b_1,...,b_n).\]
If $b_j=0$ in the string equation then the sum on the right hand side corresponding to $j$ is empty as in the following $n=0$ case.  

\subsection{\fbox{$n=0$} case}   \label{sec:n=0}

We can set $n=0$ in the string equations and the dilaton equation to get interesting results.  The string equations give vanishing results for $N^g_1$ which are useful to help calculate $N^g_1$ and in practice to check calculations.  

\begin{prop}
For $g>1$, the recursions (\ref{eq:strmon}, \ref{eq:strmon1}, \ref{eq:strmon2}) still hold when $n=0$ and the right hand side is set to zero, leading to vanishing results.
\end{prop}
\begin{proof}
This comes directly from the proof of the string equation in \cite{EOrInv}, applied to the case of $n=0$.  For $\alpha$ the poles of $ydx(z)$, $m=0$, 1 or 2 and $P^g_k$ defined by theorem 4.5 in \cite{EOrAlg} we have
\begin{eqnarray*}
&&\sum_{\alpha}\res{z=\alpha}x(z)^my(z)\omega^g_1(z)\\
&&=-\frac{1}{2}\sum_{\alpha}\res{z=\alpha}\frac{x(z)^m}{dx(z)}(-2y(z)dx(z)\omega^g_1(z)+\sum_{h=0}^g\omega^h_1(z)\omega^{g-h}_1(z)+\omega^{g-1}_2(z,z))\\
&&=\frac{1}{2}\sum_{a=\pm1}\res{z=a}\frac{x(z)^m}{dx(z)}(-2y(z)dx(z)\omega^g_1(z)+\sum_{h=0}^g\omega^h_1(z)\omega^{g-h}_1(z)+\omega^{g-1}_2(z,z))\\
&&=\frac{1}{4}\sum_{a=\pm1}\res{z=a}\frac{x(z)^m}{dx(z)}(-2y(z)dx(z)\omega^g_1(z)+\sum_{h=0}^g\omega^h_1(z)\omega^{g-h}_1(z)+\omega^{g-1}_2(z,z))\\
&&+\frac{1}{4}\sum_{a=\pm1}\res{z=a}\frac{x(z)^m}{dx(z)}(-2y(\frac{1}{z})dx(z)\omega^g_1(\frac{1}{z})+\sum_{h=0}^g\omega^h_1(\frac{1}{z})\omega^{g-h}_1(\frac{1}{z})+\omega^{g-1}_2(\frac{1}{z},\frac{1}{z}))\\
&&=\frac{1}{4}\sum_{a=\pm1}\res{z=a}\frac{x(z)^m}{dx(z)}(P^g_0(x(z))dx^2(z))\\
&&=0
\end{eqnarray*}
Note that in the general proof for arbitrary $n$, the terms added in line two are more complex recursions with Bergmann kernel terms present, contributing extra residues.
\end{proof}
\begin{cor}
For the plane curve $x=z+1/z$, $y=y(z)$
\begin{align*}
\dd y(\dd)\left\{mN^g_1(m)\right\}_{m=-1}\hspace{-1mm}&=0\\
(1+\dd^2)y(\dd)\left\{mN^g_1(m)\right\}_{m=-1}&\hspace{-1mm}=0\\
(1-\dd^2)^2y(\dd)\left\{mN^g_1(m)\right\}_{m=-2}\hspace{-1mm}&=0
\end{align*}
\end{cor}
\noindent When $n=0$ the dilaton equation is used to define $F^g$.
\begin{proof}[Proof of Corollary~\ref{th:sympl}]
We need to show that
\[F^g=\frac{1}{2-2g}(1-\dd^2)y(\dd)\left\{N^g_1(m)\right\}_{m=-1}\]
for $g\geq 2$.  The definition of the symplectic invariant uses the $n=0$ version of  the dilaton equation
\begin{equation}  \label{eq:sympl}
\sum_{\alpha}\res{z=\alpha}\Phi(z)\omega_1^g(z)=:(2g-2)F^g
\end{equation}
so we can simply apply the proof of equation~(\ref{eq:dil}) to the left hand side of this and the result follows.
\end{proof}

\section{Examples}   \label{sec:example}

This section serves a few purposes.  It describes different interesting examples and gives small $(g,n)$ polynomials in each case.  These examples led to some of the general theorems in this paper and give checks of all of the theorems.  It is difficult to calculate infinite families of invariants so this section also gives examples of symplectic invariants $F^g$ (corresponding to $n=0$) which are known for all $g$.

\subsection{\fbox{$y(z)=z$} \label{sec:lattice} Branched covers and discrete surfaces}  The curve 
\begin{equation}
C=\begin{cases}x=z+1/z\\ 
y=z
\end{cases}
\end{equation} 
gives rise to two rather different counts.  

Expand the invariants $\omega^g_n$ of the curve $(x,y)=(z+1/z,z)$ in $x_i$ around $x_i=\infty$ for $i=1,...,n$.  Eynard and Orantin \cite{EOrAlg} show that this gives a generating function for counting connected orientable discrete surfaces of genus $g$ with $n$ polygonal faces and a marked edge on each face.  The coefficient of $\prod x_i^{-(l_i+1)}$ in the expansion of $\omega^g_n$ counts the surfaces consisting of $l_i$-sided polygons, $i=1,...,n$.

The associated polynomials $N^g_n$, obtained by expanding $\omega^g_n$ in $z_i$ around $z_i=0$, were shown in \cite{NorStr} to count connected topologically inequivalent genus $g$ branched covers of $S^2$ branched over $0$, $1$ and $\infty$ with ramification $(b_1,...,b_n)$ over $\infty$, ramification $(2,2,...,2)$ over $1$ and ramification greater than 1 at all points above $0$. They are counted in such a way that each branched cover contributes one divided by the order of its group of automorphisms.   Equivalently, they count surfaces with $n$ polygonal faces of {\em lengths} $b_1,...,b_n$.  Although this resembles the count above, it is quite different.  The number $N^g_n(b_1,...,b_n)$ is presented in \cite{NorCou} in terms of counting lattice points inside integral convex polytopes depending on $(b_1,...,b_n)$ which make up a cell decomposition of $\modm_{g,n}$, the moduli space of genus $g$ curves with $n$ labeled points.  A brief description follows.

Let $\modm_{g,n}$ be the moduli space of genus $g$ curves with $n$ labeled points.  The {\em decorated} moduli space $\modm_{g,n}\times\br^n_+$ equips the labeled points with positive numbers $(b_1,...,b_n)$ \cite{PenDec}.   It has a cell decomposition due to Penner, Harer, Mumford and Thurston 
\begin{equation}  \label{eq:cell}
\modm_{g,n}\times\br^n_+\cong\bigcup_{\Gamma\in \fat}P_{\Gamma}
\end{equation}
where the indexing set $\fat$ is the space of labeled fatgraphs of genus $g$ and $n$ boundary components.  A {\em fatgraph} is a graph $\Gamma$ with vertices of valency $>2$ equipped with a cyclic ordering of edges at each vertex.    The cell decomposition (\ref{eq:cell}) arises by the existence and uniqueness of meromorphic quadratic differentials with foliations having compact leaves, known as Strebel differentials which can be described via labeled fatgraphs with lengths on edges.  Restricting this homeomorphism to a fixed $n$-tuple of positive numbers $(b_1,...,b_n)$ yields a space homeomorphic to $\modm_{g,n}$ decomposed into compact convex polytopes 
\[P_{\Gamma}(b_1,...,b_n)=\{{\bf x}\in\br_+^{E(\Gamma)}|A_{\Gamma}{\bf x}={\bf b}\}\]
where ${\bf b}=(b_1,...,b_n)$ and $A_{\Gamma}:\br^{E(\Gamma)}\to\br^n$ is the incidence matrix that maps an edge to the sum of its two incident boundary components.

When the $b_i$ are positive integers the polytope $P_{\Gamma}(b_1,...,b_n)$ is an integral polytope and we define $N_{\Gamma}(b_1,...,b_n)$ to be its number of positive integer points.  The weighted sum of $N_{\Gamma}$ over all labeled fatgraphs of genus $g$ and $n$ boundary components is the piecewise polynomial \cite{NorCou}
\[ N^g_n(b_1,...,b_n)=\sum_{\Gamma\in \fat}\frac{1}{|Aut \Gamma|}N_{\Gamma}(b_1,...,b_n)\]
The top homogeneous degree terms of the polynomials $N^g_{n,k}$ ($k$ even) representing $N^g_n$ coincides with Kontsevich's volume polynomial \cite{KonInt}
\[\displaystyle V^g_n(b_1,...,b_n)=\sum_{\Gamma\in \fat}\frac{1}{|Aut \Gamma|}V_{\Gamma}(b_1,...,b_n)\]
where $V_{\Gamma}(b_1,...,b_n)$ is the volume of $P_{\Gamma}(b_1,...,b_n)$ induced from the Euclidean volumes on $\br^{E(\Gamma)}$ and $\br^n$.  Kontsevich showed that the volume polynomial is a generating function for intersection numbers over the moduli space, so this gives an alterative proof of Theorem~\ref{th:inter} in this case.

Each integral point in the polytope $P_{\Gamma}(b_1,...,b_n)$ corresponds to a Dessin d'en\-fants defined by Grothendieck \cite{GroEsq} which is a branched cover of $S^2$ branched over $0$, $1$ and $\infty$ with ramification $(b_1,...,b_n)$ over $\infty$, ramification $(2,2,...,2)$ over $1$ and ramification greater than 1 at all points above $0$ as defined above.  

\begin{table}[h]  \label{tab:poly}
\caption{$y=z$}
\begin{spacing}{1.4}  
\begin{tabular}{||l|c|c|c||} 
\hline\hline

{\bf g} &{\bf n}&\# odd $b_i$&$N^g_n(b_1,...,b_n)$\\ \hline

0&3&0,2&1\\ \hline
1&1&0&$\frac{1}{48}\left(b_1^2-4\right)$\\ \hline
0&4&0,4&$\frac{1}{4}\left(b_1^2+b_2^2+b_3^2+b_4^2-4\right)$\\ \hline
0&4&2&$\frac{1}{4}\left(b_1^2+b_2^2+b_3^2+b_4^2-2\right)$\\ \hline
1&2&0&$\frac{1}{384}\left(b_1^2+b_2^2-4\right)\left(b_1^2+b_2^2-8\right)$\\ \hline
1&2&2&$\frac{1}{384}\left(b_1^2+b_2^2-2\right)\left(b_1^2+b_2^2-10\right)$\\ \hline
2&1&0&$\frac{1}{2^{16}3^35}\left(b_1^2-4\right)\left(b_1^2-16\right)\left(b_1^2-36\right)\left(5b_1^2-32\right)$\\
\hline\hline
\end{tabular} 
\end{spacing}
\end{table}

The description of $N^g_n$ as counting lattice points inside cells of the moduli space enables one to prove $N^g_n(0,...,0)=\chi(\modm_{g,n})$ the orbifold Euler characteristic of the moduli space of genus $g$ curves with $n$ labeled points \cite{NorCou}.  It was further shown in \cite{NorStr} that the dilaton equation together with vanishing properties of $N^g_n$ proves $N_{g,n+1}(0,...,0)=-(2g-2+n)N_{g,n}(0,...,0)$.  In particular, when $n=0$ this gives $\chi(\modm_{g,1})=(2-2g)F^g$.  Hence the symplectic invariants for $g>1$ are given by the orbifold Euler characteristic of the moduli space of genus $g$ curves
\[F^g=\chi(\modm_g).\]





\subsection{\fbox{$y=\frac{1}{m}z^m$}  Monomials}  The example $y=z$ is the first in the family of examples
\begin{equation}   \label{eq:mon}
C=\begin{cases}x=z+1/z\\ 
y=\frac{1}{m}z^m
\end{cases}
\end{equation}

In the following we demonstrate similar behaviour between the polynomials associated to the curve $y=\frac{1}{m}z^m$ for $m>1$ and the $m=1$ case, suggesting there may be an underlying enumeration problem.
Unlike the other examples in this paper, there is not yet an interpretation of the polynomials coming from an enumeration problem. 

\begin{table}[h!]  \label{tab:poly1}
\caption{$y=\frac{1}{m}z^m$ for odd $m$}  
\begin{spacing}{1.4}  
\begin{tabular}{||l|c|c|c||} 
\hline\hline

{\bf g} &{\bf n}&\# odd $b_i$&$N^g_n(b_1,...,b_n)$\\ \hline

0&3&0,2&$1$\\ \hline
0&3&1,3&$0$\\ \hline
1&1&0&$\frac{1}{48}(b_1^2-m^2-3)$\\ \hline
1&1&1&$0$\\ \hline
0&4&0,4&$\frac{1}{4}(b_1^2+b_2^2+b_3^2+b_4^2-m^2-3)$\\ \hline
0&4&1,3&$0$\\ \hline
0&4&2&$\frac{1}{4}(b_1^2+b_2^2+b_3^2+b_4^2-m^2-1)$\\ \hline
1&2&0&$\frac{1}{384}((b_1^2+b_2^2)^2-4(m^2+2)(b_1^2+b_2^2)+3m^4+10m^2+19)$\\ \hline
1&2&1&$0$\\ \hline
1&2&2&$\frac{1}{384}(b_1^2+b_2^2-m^2-1)(b_1^2+b_2^2-3m^2-7)$\\ \hline
2&1&0&$\frac{1}{2^{16}3^35}(5b_1^8-(116m^2+196)b_1^6+(834m^4+2476m^2+2402)b_1^4$\\&&&$-(2028m^6+7908m^4+13556m^2+13116)b_1^2$\\&&&$+1305m^8+5628m^6+13494m^4+24636m^2+28665)$\\ \hline
2&1&1&$0$\\
\hline\hline
\end{tabular} 
\end{spacing}
\end{table}

Theorem~\ref{th:string0} gives the following recursions relations which generalise the $m=1$ case.  They apply to all genus and can be used to determine the genus 0 and genus 1 polynomials.  
\[
N^g_{n+1}(m,b_1,\dots b_n) =\sum_{j=1}^n\sum_{k=1}^{b_j}kN^g_n(b_1,\dots,b_n)|_{b_j=k}
\]
\begin{align*}
(m+1)N^g_{n+1}(m&+1,b_1,\dots,b_n)+(m-1)N^g_{n+1}(m-1,b_1,\dots,b_n)\\
&=2m\sum_{j=1}^n\sum_{k=1}^{b_j}kN^g_n(b_1,\dots,b_n)|_{b_j=k}-m\sum_{j=1}^nb_jN^g_n(b_1,\dots, b_n)\nonumber\end{align*}
\begin{align*}  (m-2&)N^g_{n+1}(m-2, b_1,\dots,b_n)-2mN^g_{n+1}(m, b_1,\dots,b_n)\\
&+(m+2)N^g_{n+1}(m+2, b_1,\dots,b_n)
=m\sum_{j=1}^n\sum_{k=1\pm b_j}kN^g_n(b_1,\dots,b_n)|_{b_j=k}\nonumber
\end{align*}
\[
N^g_{n+1}(m+1,b_1,...,b_n)-N^g_{n+1}(m-1,b_1,...,b_n)=m(2g-2+n)N^g_n(b_1,...,b_n)
\]

\begin{table}[h!]  \label{tab:poly2}
\caption{$y=\frac{1}{m}z^m$ for even $m$}
\begin{spacing}{1.4}  
\begin{tabular}{||l|c|c|c||} 
\hline\hline

{\bf g} &{\bf n}&\# odd $b_i$&$N^g_n(b_1,...,b_n)$\\ \hline

0&3&0,2&$0$\\ \hline
0&3&1,3&$1$\\ \hline
1&1&0&$0$\\ \hline
1&1&1&$\frac{1}{48}(b_1^2-m^2-3)$\\ \hline
0&4&0,4&$\frac{1}{4}(b_1^2+b_2^2+b_3^2+b_4^2-m^2)$\\ \hline
0&4&1,3&$0$\\ \hline
0&4&2&$\frac{1}{4}(b_1^2+b_2^2+b_3^2+b_4^2-m^2-2)$\\ \hline
1&2&0&$\frac{1}{384}(b_1^2+b_2^2-m^2)(b_1^2+b_2^2-3m^2-8)$\\ \hline
1&2&1&$0$\\ \hline
1&2&2&$\frac{1}{384}((b_1^2+b_2^2)^2-4(m^2+2)(b_1^2+b_2^2)+3m^4+8m^2+12)$\\ \hline
2&1&0&$0$\\ \hline
2&1&1&$\frac{1}{2^{16}3^35}(5b_1^8-(116m^2+196)b_1^6+(834m^4+2356m^2+1982)b_1^4$\\&&&$-(2028m^6+7428m^4+10796m^2+3396)b_1^2$\\&&&$+1305m^8+5268m^6+10914m^4+11436m^2+1605)$\\
\hline\hline
\end{tabular} 
\end{spacing}
\end{table}
The following vanishing result is proved by considering the Euler characteristic of connected branched covers of the two-sphere.
\begin{prop}[\cite{NorStr}]    \label{th:van2}
For $x=z+1/z$ and $y=z$, if 
\[\displaystyle\sum_{i=1}^nb_i\leq-2\chi=4g-4+2n,\quad b_i>0\] 
then $N^g_n(b_1,...,b_n)=0$.  
\end{prop}
In particular, $N^g_n(1,1,...,1)=0$ which generalises as follows.
\begin{prop}   \label{th:van4}
For $x=z+1/z$ and $\displaystyle y=\frac{z^m}{m}$, $N^g_n(m,1,1,...,1)=0$. 
\end{prop}
\begin{proof}
We have 
\begin{equation*}
N^g_n(m,b_1,...,b_n)=\sum_{j=1}^n\hspace{-3mm}\sum_{\tiny\begin{array}{c}k=1\\k\not\equiv b_j(2)\end{array}}^{b_j}\hspace{-3mm}kN^g_n(b_1,\dots,b_n)|_{b_j=k}
\end{equation*}
As in Section~\ref{sec:n=0} which makes sense of the string equation for $n=0$, setting $b_1,...,b_n=1$ leaves an empty sum on the right hand side, hence it is zero.
\end{proof}

\subsection{\fbox{$y=\ln(z)$}  Partitions with Plancherel measure} \label{sec:planch}

For a partition $\lambda_1\geq\lambda_2\geq...\geq\lambda_N\geq 0$ define $|\lambda|=\sum\lambda_i$ and $n(\lambda)=\#\{i:\lambda_i\neq 0\}$.  The Plancherel measure on partitions of size $|\lambda|=N$ uses the dimensions of irreducible representations of $S_N$, labeled by partitions $\lambda$, satisfying $\sum_{|\lambda|=N}\dim(\lambda)^2=N!$.  The partition function 
\[Z_N(Q)=\sum_{n(\lambda)\leq N}\left(\frac{\dim\lambda}{|\lambda|!}\right)^2Q^{2|\lambda|}\]
 is related to the symplectic invariants of the curve \cite{EynOrd}
\begin{equation}  \label{eq:lnz}
C=\begin{cases}x=z+1/z\\ 
y=\ln{z}
\end{cases}
\end{equation}
via the asymptotic expansion as $Q\to\infty$
\[\ln Z_N(Q)=\sum_gQ^{2-2g}F^g.\]
For $N\to\infty$, $\exp(-Q^2)Z_N(Q)\to 1$ so 
\[F^g=\delta_{g,0}.\]
\begin{table}[h!]  \label{tab:poly3}
\caption{$y=\ln{z}$}
\begin{spacing}{1.4}  
\begin{tabular}{||l|c|c|c||} 
\hline\hline

{\bf g} &{\bf n}&\# odd $b_i$&$N^g_n(b_1,...,b_n)$\\ \hline

0&3&0,2&$0$\\ \hline
0&3&1,3&$1$\\ \hline
1&1&0&$0$\\ \hline
1&1&1&$\frac{1}{48}(b_1^2-3)$\\ \hline
0&4&0,4&$\frac{1}{4}(b_1^2+b_2^2+b_3^2+b_4^2)$\\ \hline
0&4&1,3&$0$\\ \hline
0&4&2&$\frac{1}{4}(b_1^2+b_2^2+b_3^2+b_4^2-2)$\\ \hline
1&2&0&$\frac{1}{384}(b_1^2+b_2^2-8)(b_1^2+b_2^2)$\\ \hline
1&2&1&$0$\\ \hline
1&2&2&$\frac{1}{384}(b_1^2+b_2^2-6)(b_1^2+b_2^2-2)$\\ \hline
2&1&0&$0$\\ \hline
2&1&1&$\frac{1}{2^{16}3^35}(b_1^2-1)^2(5b_1^4-186b_1^2+1605)$\\
\hline\hline
\end{tabular} 
\end{spacing}
\end{table}

The curve (\ref{eq:lnz}) can be seen as the $m=0$ case of the previous family of examples (\ref{eq:mon}). 
The polynomials satisfy the following recursions relations.
\begin{align*}
N^g_{n+1}(0,b_1,\dots b_n) &=\sum_{j=1}^n\sum_{k=1}^{b_j}kN^g_n(b_1,\dots,b_n)|_{b_j=k}
\\
N^g_{n+1}(1,b_1,\dots,b_n)
&=\sum_{j=1}^n\sum_{k=1}^{b_j}kN^g_n(b_1,\dots,b_n)|_{b_j=k}+\frac{\chi-|b|}{2}N^g_n(b_1,\dots, b_n)\end{align*}
for $\chi=2-2g-n$ and $|b|=\sum_{j=1}^nb_j$.  

In all calculated cases, for $x=z+1/z$ and $y=z^m/m$ or $y=\ln z$, the genus 0 invariants $N^0_n$ take integral values.  When $m=1$, the genus 0 invariants are proven to be integral, and based on observation it seems reasonable to conjecture that the genus 0 invariants are integral for $m>1$.  This lends further evidence that there may be an underlying enumeration problem.  In general for Gromov-Witten or moduli space calculations, the invariants are rational, and integral in genus 0.

\subsubsection{\fbox{$y=\ln(z)+cz$}  Expectation values of functions on partitions}
The following natural function on partitions 
\[C_k(\lambda)=\sum_{i=1}^N\left(\lambda_i-i+\frac{1}{2}\right)^k-\left(-i+\frac{1}{2}\right)^k+(1-2^{-k})\zeta(-k)\]
is called a Casimir in \cite{EynOrd}, and a shifted symmetric polynomial, notated by $p_k$, in \cite{OPaGro}.  The expectation values of $C_2$ with respect to the Plancherel measure on partitions are encoded in a generating function  
which gives rise to the following curve \cite{EynOrd} (after making some coordinate changes.)
\begin{equation}
C=\begin{cases}x=z+1/z\\ 
y=\ln{z}+cz
\end{cases}
, \quad  \text{$c$ is a constant}
\end{equation}
The curve has a symmetry
\[(z,c)\mapsto(-z,-c)\]
since $x\mapsto -x$ and $y\mapsto y+\ln(-1)$ which implies that when $n+k$ is even $N^g_{n,k}$ is a function of $c^2$ and when $n+k$ is odd $cN^g_{n,k}$ is a function of $c^2$.  In particular, $F^g$ is a function of $c^2$.

In the limit $c\to\infty$, $y(z)/c\to z$ (uniformly in neighbourhoods of $\pm 1$) and this allows us to deduce that the symplectic invariants for $g>1$ are asymptotic to the symplectic invariants for $y=z$, i.e. 
\[F^g\sim\chi(\modm_g)c^{2-2g},\quad c\to\infty.\]  
The polynomials interpolate between the $y=\ln{z}$ and $y=z$ cases.  More precisely, 
\[\lim_{c\to0}N^g_n=N^g_n[y(z)=\ln{z}],\quad\quad\lim_{c\to\infty}c^{2g-2+n}N^g_n=N^g_n[y(z)=z].\]

\begin{table}[h!]  \label{tab:poly4}
\caption{$y=\ln{z}+cz$,\quad $D=y'(1)y'(-1)=c^2-1$}
\begin{spacing}{1.4}  
\begin{tabular}{||l|c|c|c||} 
\hline\hline

{\bf g} &{\bf n}&\# odd $b_i$&$N^g_n(b_1,...,b_n)$\\ \hline

0&3&0,2&$\frac{c}{D}$\\ \hline
0&3&1,3&$\frac{-1}{D}$\\ \hline
1&1&0&$\frac{c}{48D^2}(Db_1^2-4c^2+2)$\\ \hline
1&1&1&$\frac{1}{48D^2}(-Db_1^2+5c^2-3)$\\ \hline
0&4&0,4&$\frac{1}{4D^3}((c^4-1)(b_1^2+b_2^2+b_3^2+b_4^2)-4c^4)$\\ \hline
0&4&1,3&$\frac{-c}{2D^3}(D(b_1^2+b_2^2+b_3^2+b_4^2)-3c^2+1)$\\ \hline
0&4&2&$\frac{1}{4D^3}((c^4-1)(b_1^2+b_2^2+b_3^2+b_4^2)-2c^4-4c^2+2)$\\ \hline
1&2&0&$\frac{1}{384D^4}((c^2+1)D^2(b_1^2+b_2^2)^2+32c^6$\\ &&&$
-4D(3c^4+3c^2-2)(b_1^2+b_2^2))$\\ \hline
2&1&0&$\frac{c}{2^{16}3^35D^7}(5(c^2+3)D^4b_1^8-8(39c^4+136c^2-59)D^3b_1^6$\\ &&&$+16(357c^6+1417c^4-1020c^2+254)D^2b_1^4$\\ &&&$-64(572c^8+2192c^6-1739c^4+806c^2-151)Db_1^2$\\ &&&$+6144c^7(12c^4+21c^2+2))$\\
\hline\hline
\end{tabular} 
\end{spacing}
\end{table}

\subsubsection{$q$-deformed partition}
The Plancherel weights $P(\lambda)=(\dim{\lambda}/|\lambda|!)^2$ used in Section~\ref{sec:planch} can be replaced by $P_q(\lambda)=(\dim_q{\lambda}/[|\lambda|]!)^2$, $q$-deformed Plancherel weights  where $q$-numbers are used in place of integer combinatorial expressions.  We will not give the details here since we will only use the spectral curve calculated in \cite{EynOrd} for this case.
\begin{equation}
C=\begin{cases}x=(1-\frac{z}{z_0})(1-\frac{1}{zz_0})\\ 
y=\frac{1}{x(z)}\left(-\ln{z}+\frac{1}{2}\ln(\frac{1-\frac{z}{z_0}}{1-\frac{1}{zz_0}})\right)
\end{cases},
\end{equation}
The curve has a symmetry
\[(z,z_0)\mapsto(-z,-z_0)\]
which implies that when $k$ is even $N^g_{n,k}$ is a function of $z_0^2$ and when $k$ is odd $z_0N^g_{n,k}$ is a function of $z_0^2$.  In particular, $F^g$ is a function of $z_0^2$.   Put $z_0^2=1-e^t$ so $F^g$ is a function of $e^{t}$.

It was proven in \cite{EynOrd} that
\[ F^g=c_g+\sum_{d>0}|\chi(\modm_g)|\frac{d^{2g-3}}{(2g-3)!}e^{-td}\]
for a constant $c_g$.  The Euler characteristic $\chi(\modm_g)$ arises in some sense by coincidence as an expression involving Bernoulli numbers, with no geometric relation to the moduli space.  The asymptotics of $F^g$ as $t\to 0$ can explain the appearance of $\chi(\modm_g)$ geometrically.  As $t\to 0$, or equivalently $z_0\to 0$,
\[x\sim\frac{1}{z_0^2}-\frac{1}{z_0}\left(z+\frac{1}{z}\right)\equiv-\frac{1}{z_0}\left(z+\frac{1}{z}\right),\quad y\sim \frac{z_0^3}{2}\left(z-\frac{1}{z}\right)\equiv z_0^3z\]
where $x$ has been shifted by a constant and $y$ has been adjusted by a linear function in $x$ which does not affect the invariants.  Under this scaling the invariants scale by
\[\omega^g_n\sim (-1)^{2-2g-n}z_0^{4-4g-2n}\omega^g_n[y=z]\]
and in particular 
\[F^g\sim z_0^{4-4g}F^g[y=z] =|\chi(\modm_g)|t^{2-2g},\quad t\to 0\]
where $F^g[y=z]=\chi(\modm_g)$ follows from the relation of the case $x=z+1/z$, $y=z$ to counting lattice points in the moduli space of curves described in Section~\ref{sec:lattice}.

\begin{table}[h]  \label{tab:poly5}
\caption{$q$-deformed partitions}
\begin{spacing}{1.4}  
\begin{tabular}{||l|c|c|c||} 
\hline\hline

{\bf g} &{\bf n}&\# odd $b_i$&$N^g_n(b_1,...,b_n)$\\ \hline

0&3&0,2&$\frac{-1}{z_0^2}(3z_0^2+1)$\\ \hline
0&3&1,3&$\frac{1}{z_0}(z_0^2+3)$\\ \hline
1&1&0&$\frac{-1}{48z_0^2}((3z_0^2+1)b_1^2-2z_0^2-4)$\\ \hline
1&1&1&$\frac{1}{48z_0}((z_0^2+3)b_1^2-3z_0^2-3)$\\ \hline
0&4&0,4&$\frac{1}{4z_0^4}((1+z_0^2)(z_0^4+14z_0^2+1)(b_1^2+b_2^2+b_3^2+b_4^2)-4)$\\ \hline
0&4&1,3&$\frac{-1}{2z_0^3}((z_0^2+3)(3z_0^2+1)(b_1^2+b_2^2+b_3^2+b_4^2)-z_0^4+4z_0^2-3)$\\ \hline
1&2&0&$\frac{1}{384z_0^4}((z_0^2+1)(z_0^4+14z_0^2+1)(b_1^2+b_2^2)^2-32z_0^2+32$\\ &&&$-4(2z_0^6+3z_0^4+8z_0^2+3)(b_1^2+b_2^2))$\\ \hline
2&1&0&$\frac{1}{2^{16}3^35z_0^6}(-5(3z_0^2+1)(3z_0^6+27z_0^4+33z_0^2+1)b_1^8$\\ &&&$+(952z_0^8+224z_0^6+336z_0^4+3808z_0^2+312)b_1^6$\\ &&&$+(-3680z_0^8+11360z_0^6+20688z_0^4-13952z_0^2-5712)b_1^4$\\ &&&$+(-5696z_0^8+5120z_0^6-38592z_0^4-21248z_0^2+36608)b_1^2$\\ &&&$-36864(z_0^2-2)(z_0^2-1))$\\
\hline\hline
\end{tabular} 
\end{spacing}
\end{table}

\subsection{\fbox{$y=\displaystyle{\frac{t}{t-\gamma z}}$} Discrete surfaces} \label{sec:disc}
The example in Section~\ref{sec:lattice} produced a generating function for counting connected orientable discrete surfaces of genus $g$ with $n$ polygonal faces and a marked edge on each face.  That is a special case of more general connected orientable discrete surfaces of genus $g$ with $n$ marked polygonal faces---each containing a marked edge---together with unmarked faces of perimeter greater than 2, with $n_i$ of perimeter $i$ for $i\geq 3$.  A spectral curve for this problem is constructed in \cite{EOrAlg}, where the expansion of $\omega^g_n$ in $x$ around $x=\infty$ gives a generating function that counts these discrete surfaces.  The coefficient of $\prod x_i^{-(l_i+1)}t_3^{n_3}...t_d^{n_d}$ in the expansion of $\omega^g_n$ counts the surfaces consisting of $t_3$ (unmarked) triangles, $t_4$ squares, ... and $l_i$-sided marked polygons, $i=1,...,n$.  If we set $t_k=-1$ for all $k\geq 3$ then the spectral curve is given by
\begin{equation}
C=\begin{cases}x=-2t+\gamma(z+1/z)\\ 
y=\frac{t}{t-\gamma z}
\end{cases}
, \quad \gamma^2 = t(t+1)
\end{equation}
and the generating function gives an alternating sum of numbers of discrete surfaces.

The curve has a symmetry
\[ (z,\gamma)\mapsto(-z,-\gamma)\]
which implies that when $k$ is even $N^g_{n,k}$ is a function of $\gamma^2$ and when $k$ is odd $\gamma N^g_{n,k}$ is a function of $\gamma^2$.  In particular, $F^g$ is a function of $\gamma^2$, hence a function of $t$ as shown below.  

The curve also has a less obvious symmetry
\[ (t,\gamma)\mapsto(-t-1,-\gamma)\]
which takes $x\mapsto 2-x$, and 
\[y\mapsto\frac{t+1}{t+1-\gamma z}=\frac{\gamma^2}{\gamma^2-t\gamma z}=1+\frac{tz}{\gamma-tz}\sim1+\frac{tz^{-1}}{\gamma-tz^{-1}}=1-y\]
where the $\sim$ sign corresponds to reparametrising the curve by $z\mapsto z^{-1}$.  This implies that under this symmetry $N^g_n$ is invariant, respectively skew invariant, when $n$ is even, respectively $n$ is odd.  In particular, $F^g$ is invariant under $t\mapsto-t-1$.

The symplectic invariants of this curve conjecturally use the $n=0$ case of the enumerative problem above---an alternating count of surfaces consisting of $t_3$ triangles, $t_4$ squares, ... and no marked polygons.  The solution of the $n=0$ problem is
\[\sum_{k>0}\chi(\modm_{g,k})\frac{t^k}{k!}\]
where $\chi(\modm_{g,k})$ is the orbifold Euler characteristic of the moduli space of genus $g$ curves with $k$ labeled points.  This uses a cell decomposition of the moduli space where each polygon labels a cell.  The alternating sum over all cells gives the Euler characteristic $\chi(\modm_{g,k})/k!$ of the moduli space of genus $g$ curves with $k$ (unlabeled) points.

As $t\to0$,  
\[ x\sim t^{1/2}(z+1/z),\quad y\sim-t^{1/2}/z\] 
hence the curve behaves asymptotically like $x=z+1/z$, $y=-1/z$, or equivalently $y=z(=-1/z+x)$ which is the example in Section~\ref{sec:lattice} with $F^g[y=z]=\chi(\modm_g)$.  Hence $F^g\sim\chi(\modm_g)t^{2-2g}$ as $t\to 0$.  Adding the two contributions gives an asymptotic formula which is conjecturally exact:
\begin{align*}
F^g=&\chi(\modm_g)t^{2-2g}+\sum_{k>0}\chi(\modm_{g,k})\frac{t^k}{k!}\\
=&\chi(\modm_g)\left(t^{2-2g}+(-1)^k\binom{2g+k-3}{2g-3}t^k\right)\\
=&\chi\left(\modm_g)(t^{2-2g}+(t+1)^{2-2g}\right)
\end{align*}
where $\chi(\modm_g)=\zeta(1-2g)/(2-2g)$.  

It is interesting that the symmetry $t\mapsto-t-1$ of $F^g$ (proved {\em a priori}) suggests that the asymptotic part $t^{2-2g}$ determines, and is determined by, the solution of the $n=0$ enumerative problem,  $(t+1)^{2-2g}$.

\subsubsection{Relation to counting surfaces}
As discussed above, the expansion in $x$ of $\omega^g_n$ around $x=\infty$ gives a generating function for counting discrete surfaces, where the coefficient of $\prod x_i^{-(l_i+1)}t_3^{n_3}...t_d^{n_d}$ in the expansion counts surfaces consisting of $t_3$ (unmarked) triangles, $t_4$ squares, ... and $l_i$-sided marked polygons, $i=1,...,n$.  We can find this coefficient explicitly by evaluating 
\begin{eqnarray*} 
&&T^g_{l_1,...,l_n}:=(-1)^n\res{x_1=\infty}\dots \res{x_n=\infty} x_1^{l_1}\dots x_n^{l_n}\omega^g_n\\
&&=(-1)^n\prod_{i=1}^n \res{z_i=\infty} \gamma^{l_i}(z_i+\frac{1}{z_i})^{l_i}\sum_{b_1,...,b_n=1}^\infty N^g_n(b_1,...,b_n)\prod_{i=1}^nb_iz_i^{b_i-1}dz_i\\
&&=(-1)^n\prod_{i=1}^n \res{z_i=0} \gamma^{l_i}(\frac{1}{z_i}+z_i)^{l_i}\sum_{b_1,...,b_n=1}^\infty N^g_n(b_1,...,b_n)\prod_{i=1}^nb_iz_i^{b_i-1}dz_i\\
&&=(-1)^n\prod_{i=1}^n \res{z_i=0} \gamma^{l_i}\hspace{-3mm}\sum_{k_1,...,k_n=0}^{l_i}\sum_{b_1,...,b_n=1}^\infty N^g_n(b_1,...,b_n)\prod_{i=1}^nb_i\binom{l_i}{k_i}z_i^{l_i-2k_i+b_i-1}dz_i\\
&&=(-1)^n\gamma^{\sum l_i} \sum_{k_i> \frac{l_i}{2}}^{l_i}N^g_n(2k_1-l_1,...,2k_n-l_n)\prod_{i=1}^n (2k_i-l_i)\binom{l_i}{k_i}\\
\end{eqnarray*}
where the $N^g_n$'s associate to the spectral curve computed for particular choices of $t_k$.
In our example, we have set $t_3=t_4=...=-1$, so that $T^g_{l_1,...,l_n}$ counts the bias of discrete surfaces having even or odd numbers of faces:

$$T^g_{l_1,...,l_n}=\sum_{v=1}^\infty t^v\sum_{S\in M^g_n(v,l_1,...,l_n)} \frac{(-1)^{\text{\# unmarked polygons in $S$}}}{|Aut(S)|}$$
where $M^g_n(v,l_1,...,l_n)$ is the set of connected, orientable, discrete surfaces of genus $g$, constructed by connecting $n$ marked polygons of perimeter $l_1,...,l_n$ and any finite number of unmarked polygons using only $v$ vertices.  From \cite{EOrAlg} this is a finite set.  For example, 
\begin{eqnarray*}
T^{(2)}_l=\begin{cases}
0 \indent &0\leq l\leq 4\\
-8t-8t^2 \indent &l=5\\
36t+108t^2+72t^3 \indent&l=6\\
-49t-490t^2-882t^3-441t^4 \indent &l=7
\end{cases}\end{eqnarray*}
{\em Remark.} It seems that $T^g_{l_1,...,l_n}$ is polynomial in $t$.  Equivalently, most of the coefficients of powers of $t$ vanish.  This means that besides finitely many exceptions for a small number of vertices, there is a duality between discrete surfaces with an even and odd number of faces.  Furthermore, the following vanishing result shows that there are no exceptions when the $b_i$ are positive and small enough.
$$N^g_n(b_1,...,b_n)=0 \indent \text{for}\indent \sum_{i=1}^n b_i <2g+n,\quad b_i>0.$$
The vanishing result is equivalent to the fact that $\omega^g_n(z_1,...,z_n)$ vanishes at $z_i=0$ with homogeneous degree $2g$ which can be proved inductively from the recursive definition.  A similar vanishing result holds for the $y=z$ case---see Proposition~\ref{th:van2}.

\begin{table}[h]  \label{tab:poly6}
\caption{Discrete surfaces}
\begin{spacing}{1.4}  
\begin{tabular}{||l|c|c|c||} 
\hline\hline

{\bf g} &{\bf n}&\# odd $b_i$&$N^g_n(b_1,...,b_n)$\\ \hline

0&3&0,2&$\frac{1+2t}{\gamma^2}$\\ \hline
0&3&1,3&$\frac{-2}{\gamma}$\\ \hline
1&1&0&$\frac{1+2t}{48\gamma^2}(b_1^2-4)$\\ \hline
1&1&1&$\frac{-1}{24\gamma}(b_1^2-1)$\\ \hline
0&4&0,4&$\frac{1}{4\gamma^4}((8\gamma^2+1)(b_1^2+b_2^2+b_3^2+b_4^2)-8\gamma^2-4)$\\ \hline
0&4&1,3&$\frac{-1}{\gamma}(1+2t)(b_1^2+b_2^2+b_3^2+b_4^2-1)$\\ \hline
0&4&2&$\frac{1}{4\gamma^4}((8\gamma^2+1)(b_1^2+b_2^2+b_3^2+b_4^2)-8\gamma^2-2)$\\ \hline
1&2&0&$\frac{1}{384\gamma^4}(b_1^2+b_2^2-4)((8\gamma^2+1)(b_1^2+b_2^2)-16\gamma^2-8)$\\ \hline
1&2&1&$\frac{-1}{96\gamma^3}(1+2t)(b_1^2+b_2^2-1)(b_1^2+b_2^2-5)$\\ \hline
1&2&2&$\frac{1}{384\gamma^4}(b_1^2+b_2^2-2)((8\gamma^2+1)(b_1^2+b_2^2)-32\gamma^2-10)$\\ \hline
2&1&0&$\frac{1+2t}{2^{16}3^35\gamma^6}(b_1^2-4)(b_1^2-16)((80\gamma^2+5)b_1^4-(608\gamma^2+212)b_1^2$\\&&&$+1152\gamma^2+1152)$\\
\hline
2&1&1&$\frac{-1}{2^{15}3^35\gamma^5}(b_1^2-1)(b_1^2-9)((80\gamma^2+15)b_1^4-(1408\gamma^2+438)b_1^2$\\&&&$+3632\gamma^2+1575)$\\
\hline\hline
\end{tabular} 
\end{spacing}
\end{table}

\subsubsection{Quadrangulations}
One can enumerate discrete surfaces consisting of quadrilaterals by setting $t_k=0$ for $k\neq 4$ in the enumeration of discrete surfaces defined in Section~\ref{sec:disc} counts.  (There we set $t_k=-1$ for $k\geq 3$.)  The spectral curve for this problem is given in \cite{EOrAlg}
\begin{equation}
C=\begin{cases}x=z+1/z\\ 
y=tz-t_4\gamma^4z^3
\end{cases}
, \quad \gamma^2 = \frac{1-\sqrt{1-12tt_4}}{6t_4}
\end{equation}
In this case the recursion relations between polynomials are
\[
tN^g_{n+1}(1,b_S)-3t_4\gamma^4N^g_{n+1}(3,b_S) =\sum_{j=1}^n\sum_{k=1}^{b_j}kN^g_n(b_S)|_{b_j=k}
\]
\begin{align*}
2(t-t_4\gamma^4)N^g_{n+1}(2&,b_S)-4t_4\gamma^4N^g_{n+1}(4,b_S)\\
&=2\sum_{j=1}^n\sum_{k=1}^{b_j}kN^g_n(b_S)|_{b_j=k}-\sum_{j=1}^nb_jN^g_n(b_S)\nonumber\end{align*}
\begin{align*}  (t-t_4\gamma^4)N^g_{n+1}(1,b_S)&+3(t-2t_4\gamma^4)N^g_{n+1}(3,b_S)\\
&+5t_4\gamma^4N^g_{n+1}(5,b_S)
=\sum_{j=1}^n\sum_{k=1\pm b_j}kN^g_n(b_S)|_{b_j=k}\nonumber
\end{align*}
\[
-tN^g_{n+1}(0,b_S)+(t+t_4\gamma^4)N^g_{n+1}(2,b_S)-t_4\gamma^4N^g_{n+1}(4,b_S)=(2g-2+n)N^g_n(b_S)
\]
where $b_S=(b_1,\dots,b_n)$.

In the following table $y'(1)=t-3t_4\gamma^4$.
\begin{table}[h]  \label{tab:poly7}
\caption{Quadrangulations}
\begin{spacing}{1.4}  
\begin{tabular}{||l|c|c|c||} 
\hline\hline

{\bf g} &{\bf n}&\# odd $b_i$&$N^g_n(b_1,...,b_n)$\\ \hline

0&3&0,2&$\frac{1}{y'(1)}$\\ \hline
0&3&1,3&$0$\\ \hline
1&1&0&$\frac{1}{48y'(1)^2}(y'(1)b_1^2+36t_4\gamma^4-4t)$\\ \hline
1&1&1&$0$\\ \hline
0&4&0,4&$\frac{1}{4y'(1)^3}(y'(1)(b_1^2+b_2^2+b_3^2+b_4^2)+36t_4\gamma^4-4t)$\\ \hline
0&4&1,3&$0$\\ \hline
0&4&2&$\frac{1}{4y'(1)^3}(y'(1)(b_1^2+b_2^2+b_3^2+b_4^2)+30t_4\gamma^4-2t)$\\ \hline

2&1&0&$\frac{1}{2^{15}3^35y'(1)^7}(5y'(1)^4b_1^8-24y'(1)^3(13t-155t_4\gamma^4)b_1^6$\\&&& $+48y'(1)^2(119t^2-2090t_4\gamma^4t+17295t_4^2\gamma^8)b_1^4$\\&&& $-256y'(1)(143t^3-2793t^2t_4\gamma^4+25857tt_4\gamma^8-237735t_4^3\gamma^{12})b_1^2$\\&&& $+73728t^4-1548288t_4\gamma^4t^3+13934592t_4^2\gamma^8t^2$\\&&&$-36495360tt_4^3\gamma^{12}+1134673920t_4^4\gamma^{16})$\\
\hline
2&1&1&$0$\\
\hline\hline
\end{tabular} 
\end{spacing}
\end{table}

\end{document}